\documentclass{amsart}

\usepackage{fullpage}

\usepackage{graphicx}
\usepackage{bm}

\newtheorem{theorem}{Theorem}[section]

\newtheorem{lem}[theorem]{Lemma}
\newtheorem{prop}[theorem]{Proposition}

\theoremstyle{definition}

\theoremstyle{remark}

\numberwithin{equation}{section}

\newcommand{\abs}[1]{\lvert#1\rvert}

\newcommand{\sech}{{\rm sech}}
\newcommand{\ve}{\varepsilon}
\newcommand{\wt}{\widetilde}
\newcommand{\rmv}{\upsilon}
\newcommand{\rmvs}{\mbox{\small$\upsilon$}}

\begin{document}

\title[On error estimates for Galerkin finite element methods for the Camassa-Holm equation]{On error estimates for Galerkin finite element methods for the Camassa-Holm equation}


\author{D.C. Antonopoulos}
\address{\textbf{D.C. Antonopoulos:} Mathematics Department, National and Kapodistrian University of Athens, 15784 Zographou, 
Greece, Intitute of Applied and Computational Mathematics, FORTH, 70013 Heraklion, Greece}
\email{antonod@math.uoa.gr}

\author{V. A. Dougalis}
\address{\textbf{V. A. Dougalis:} Mathematics Department, National and Kapodistrian University of Athens, 15784 Zographou, 
Greece, Intitute of Applied and Computational Mathematics, FORTH, 70013 Heraklion, Greece}
\email{ doug@math.uoa.gr}

\author{Dimitrios Mitsotakis}
\address{\textbf{D.~Mitsotakis:} Victoria University of Wellington, School of Mathematics and Statistics, PO Box 600, Wellington 6140, New Zealand}
\email{dimitrios.mitsotakis@vuw.ac.nz}


%\thanks{D. Mitsotakis was supported by the Marsden Fund}

%    General info

\date{\today}



\begin{abstract}
We consider the Camassa-Holm (CH) equation, a nonlinear dispersive wave equation that models
one-way propagation of long waves of moderately small amplitude. We discretize in space the periodic initial-value problem for CH (written in its original and in system form), using the standard
Galerkin finite element method with smooth splines on a uniform mesh, and prove optimal-order 
$L^{2}$-error estimates for the semidiscrete approximation. We also consider an 
initial-boundary-value problem on a finite interval for the system form of CH and analyze 
the convergence of its standard Galerkin semidiscretization. Using the fourth-order accurate,
explicit, ``classical'' Runge-Kutta scheme for time-stepping, we construct a highly accurate,
stable, fully discrete scheme that we employ in numerical experiments to approximate solutions of CH,
mainly smooth travelling waves and nonsmooth solitons of the `peakon' type.
\end{abstract}

\maketitle

%%%%%%%%%%%%%%%%%%%%%%%%%%%%%%%%%%%%%%%%%%%%%%%%%%%%%%%%%%%%%%%%%
\section{Introduction}
In this paper we analyze standard Galerkin finite element approximations to the 
\textit{Camassa-Holm} (CH) equation
\begin{equation}
u_{t} + 2ku_{x} + 3uu_{x} - u_{xxt} = 2u_{x}u_{xx} + uu_{xxx}\ .
\label{eq11}
\end{equation}
The equation may be derived as a bi-Hamiltonian integrable system by the method of 
\cite{ff} (cf. \cite{fo}, \cite{fu}). It is named after R. Camassa and D.D. Holm, who derived it
in \cite{ch} from the Euler equations of water-wave theory as a model for the unidirectional
propagation of long waves on the surface of an ideal fluid in a uniform horizontal channel.
In \eqref{eq11}, which is written in nondimensional, unscaled variables, $u=u(x,t)$ is 
proportional to the depth-averaged horizontal velocity of the fluid at $x$, $t$, that
are proportional to position along the channel and time, respectively. The equation was further
studied from this viewpoint in \cite{j1}, and was rigorously justified in \cite{cl} as an
$O(\sigma^{4})$-accurate unidirectional approximation to the Euler equations when the scaling 
parameters satisfy $\sigma\ll 1$, $\epsilon=O(\sigma)$. Here $\epsilon=a/h_{0}$,
$\sigma^{2} = h_{0}^{2}/\lambda^{2}$, $h_{0}$ is the depth of the channel, $a$ is a
typical surface wave amplitude, and $\lambda$ a typical wavelength. In scaled, nondimensional
variables such a derivation leads e.g. to the equation
\begin{equation}
\rmv_{\tau} + \rmv_{y} + \epsilon \rmv\rmv_{y} 
- \sigma^{2}\rmv_{yy\tau} - 
\tfrac{\epsilon\sigma^{2}}{3}(2\rmv_{y}\rmv_{yy} 
+ \rmv\rmv_{yyy}) = O(\sigma^{4})\ ,
\label{eq12}
\end{equation}
from which we may recover \eqref{eq11} by replacing the right-hand side by zero and making the 
change of variables $u(x,t) = 2k\epsilon \rmv(y,\tau)/3$, 
$y = \sigma x$, $\tau=2k\sigma t$. Hence,
when compared with the scaled BBM equation, for which $\sigma \ll 1$, $\epsilon = O(\sigma^{2})$,
the equation \eqref{eq12} approximates the Euler equations to the same order of formal accuracy,
possesses an additional nonlinear dispersive term, and remains valid in the larger wave-amplitude
regime $\epsilon = O(\sigma)$. \par
The Camassa-Holm equation has been studied extensively and shown to possess global smooth solutions, 
but also solutions that develop `breaking wave' singularities in finite time; in such singular
solutions the derivative $u_{x}$ blows up while $u$ remains bounded. As expected from its complete
integrability properties, the equation possesses solitons, i.e. solitary waves that interact
cleanly. It should be noted that \eqref{eq11} may be written in system form e.g. as
\begin{equation}
\begin{aligned}
& m = u - u_{xx}\ ,\\
& m_{t} + 2ku_{x} + um_{x} + 2u_{x}m = 0\ .
\end{aligned}
\label{eq13}
\end{equation}
The literature on the well-posedness of the initial-value problem (ivp) for the CH is large; here
we will just mention some basic references. (A survey of results up to 2003 may be found in 
\cite{m}. It should be noted that in some of the works to be mentioned below $k$ is taken
equal to zero in \eqref{eq11} and in \eqref{eq13}; for the significance of $k$ in the context
of travelling-wave solutions, see the remarks below.) In \cite{ce1} Constantin and Escher
established the existence of local in time solutions in $H^{3}$ for \eqref{eq11} with $k=0$, and 
provided sufficient conditions on the initial data that lead to global solutions or to the
emergence of singularities in finite time. Analogous results for the periodic ivp for the CH, written in the system form \eqref{eq13} with $k=0$, were established in \cite{cper} in the Sobolev 
space $H^{2}_{per}$ of periodic functions. In \cite{lo} Li and Olver studied the well-posedness
of the ivp for a slightly generalized form of \eqref{eq11} in $H^{s}$ for $s > 3/2$,
put forth alternative conditions to those of \cite{ce1} for global existence and finite-time blow 
up of solutions, and also studied the existence of weaker solutions in $H^{s}$ for $1<s\leq 3/2$.
The development of singularities has been studied in several papers in addition to those already
mentioned, e.g. in \cite{ce2}, \cite{m}, and in \cite{cl}; in the latter reference the large-time
$(O(T/\epsilon))$ existence of solutions and the breaking of waves of the free surface equation
(in addition to the equation of the velocity $u$) have also been analyzed. More recently, issues of
existence of global generalized solutions, possibly after wave-breaking occurs, have been studied,
cf. e.g. \cite{bc}. We also note that the well-posedness of initial-boundary value problems
(ibvp's) for CH (usually written in the system form \eqref{eq13}) has also been studied; see e.g. 
\cite{k} for the local well-posedness in $H^{4}$ of the ibvp for \eqref{eq13} with $k=0$ on a
finite interval with boundary conditions $u=m=0$ at the endpoints, and also \cite{ey},
where the well-posedness of the quarter-plane and the two-point ibvp has been studied in 
more general spaces. \par
Since CH is completely integrable it is amenable to study in terms of the Inverse Scattering 
Transform (IST), cf. e.g. \cite{cscat}, \cite{c-le}, and their references. A complete 
classification of its travelling-wave solutions has been carried out in \cite{le}; such solutions
may be smooth or nonsmooth, in the latter case satisfying the equation in the sense of 
distributions. Its solitary waves are solitons. If $k\ne 0$ CH possesses smooth solitons and
N-solitons; these have been explicitly constructed, and their interactions studied by IST; cf. 
\cite{j2}, and especially \cite{pai,paii,paiii,p}. For $k=0$ the equation does not have
smooth solitary waves; it has however generalized soliton solutions (`peakons') of the form
$u(x,t) = c\exp(-\lvert x - ct\rvert)$ for any $c >0$; the study of existence and interactions of 
peakons has attracted a lot of attention, starting with \cite{ch}. (As previously mentioned,
other types of singular travelling-wave solutions have been identified in \cite{le}.) Finally, 
 let  us mention that the stability of smooth solitons and peakons has been established in
\cite{c-si}, \cite{c-sii}; see also \cite{c-mo}. \par
A note on the constant $k$ that appears in \eqref{eq11}: As the change of variables that 
connects the scaled version \eqref{eq12} of CH with \eqref{eq11} implies, $k$ is related to the 
speed of e.g. the travelling-wave solutions of \eqref{eq12}, and, therefore, for physical water
waves it is necessary that $k\ne 0$. We may assume that $k>0$, since, otherwise, the change of
variables $U(X,T)=-u(x,t)$, $X = -x$, $T=t$ leads to a CH equation with $-2kU_{X}$
instead of $2ku_{x}$. It has also been noticed since \cite{ch} that the change of variables
$w(z,s) = u(x,t) + k$, $z = x + kt$, $s = t$, transforms \eqref{eq11} into
\begin{equation}
w_{s} + 3ww_{z} - w_{zzs} = 2w_{z}w_{zz} + ww_{zzz}\ ,
\label{eq14}
\end{equation}
i.e. the CH without the $2ku_{x}$ term. (We will usually write \eqref{eq14} using the variables
$u$, $x$, $t$ and call it, cf. \cite{pai}, reduced CH equation (RCH).) As a consequence of this
transformation and previous remarks on the existence of solitary waves of CH and RCH, it follows
that RCH possesses smooth travelling-wave solutions decaying at infinity to $k$, and CH with
$k\ne 0$ possesses peakon-type travelling waves decaying to $-k$. \par
Since the derivation of CH, numerical methods have been used in an exploratory fashion to 
illuminate aspects of the generation and interactions of its solutions, see e.g. \cite{chh}.
The ensuing numerical literature is large; we confine ourselves to mentioning a few papers
that the reader may consult along with their references. Among works that focus on the analysis
of numerical methods and error estimates cf. e.g. \cite{kr} (spectral methods), \cite{hr}
and \cite{ckr} (finite difference methods), and \cite{clp1} (particle methods). Works that
focus on the construction of numerical schemes and numerical experimentation include e.g.
\cite{kl} (spectral methods), \cite{xs} and \cite{lx} (DG methods), and \cite{clp2}
(particle methods). \par
In the paper at hand we consider Galerkin finite element methods for the numerical solution of
CH. In Sections \ref{sec2} and \ref{sec3} we provide the necessary background and notation and prove 
optimal-order-of-convergence $L^{2}$-error estimates for the standard Galerkin semidiscrete 
approximation for the periodic ivp for \eqref{eq11} using smooth splines on a uniform mesh.
The proof is effected by comparing the semidiscrete approximation with the Thom{\'e}e-Wendroff
quasiinterpolant, \cite{tw}. In Section \ref{sec4} we consider the analogous periodic ivp for the system
\eqref{eq13} and prove $L^{2}$-error estimates of optimal rate of convergence for the 
semidiscrete standard Galerkin approximations of $m$ and $u$ using again smooth periodic splines
on a uniform mesh and employing similar error estimation techniques. We refer to this numerical
scheme as the `modified' Galerkin method and note that it requires splines of order $r\geq 2$
(i.e. continuous piecewise polynomial functions of degree $r-1\geq 1$), thus allowing the use of
piecewise linear continuous functions, as opposed to the standard Galerkin scheme for \eqref{eq11}
that requires $r\geq 3$, i.e. using at least $C^{1}$ quadratics due to the presence of the 
$uu_{xxx}$ term. In Section \ref{sec5} we consider the ibvp for the system \eqref{eq13} on a finite interval
with zero Dirichlet boundary conditions for $m$ and $u$ at the endpoints. We study the convergence
of the semidiscrete approximation of the problem by the standard Galerkin method on a 
quasiuniform mesh with piecewise polynomial functions of degree $r-1$ with $r\geq 2$. Using
energy estimates, properties of the $H^1$- and $L^{2}$-projections onto the finite element
spaces and duality techniques we prove a suboptimal $L^{2}$-error estimate of $O(h^{r-1})$ for the
approximation of $m$, and an optimal $H^{1}$-estimate of $O(h^{r})$ and a suboptimal 
$L^{2}$-estimate of $O(h^{r-1/2})$ for the approximation of $u$. We took $k=0$ for simplicity
in Sections \ref{sec3}-\ref{sec5}; the inclusion of the $2ku_{x}$ term poses no difficulty and the results are the
same. \par
In Section \ref{sec6} we present the results of some numerical experiments that we performed using the
Galerkin methods described above in space and discretizing the ivp's for the attendant ode systems 
in the temporal variable by the classical, explicit, fourth-order-accurate Runge-Kutta scheme. The
resulting fully discrete methods are stable under a mild restriction on the Courant number 
$\Delta t/h$, as the problem is not very stiff due to the presence of the BBM type term $-u_{xxt}$
in the left-hand side of \eqref{eq11}. In the numerical experiments we check the spatial order 
of convergence of the schemes in the case of smooth solutions and investigate experimentally 
the order of convergence for peakons. We check the conservation 
properties of the discrete schemes and study the fidelity of the numerical approximations of the
travelling waves by computing their amplitude, phase, shape, and speed errors. Finally, we study
the generation and interactions of peakons and compare the accuracy of our schemes with 
that of other numerical methods in the literature. \par
In this paper we denote, for integer $k\geq 0$, by $H^{k}=H^{k}(0,1)$ the usual $L^{2}$-based
Sobolev spaces on $[0,1]$ by  $\mathring{H}^{1}$ the subspace of $H^{1}$ consisting of functions
with zero boundary conditions, and by $H_{per}^{k}$ the subspace of $H^{k}$ consisting of
1-periodic functions; in  all cases the corresponding norms are denoted by $\|\cdot\|_{k}$.
We let $C^{k}$ denote the $k$-times continuously differentiable functions on $[0,1]$ and
$C_{per}^{k}$ the 1-periodic such functions. The inner product on $L^{2}=L^{2}(0,1)$ is denoted
by $(\cdot,\cdot)$ and the corresponding norm simply by $\|\cdot\|$. The norms of 
$W_{\infty}^{k}=W_{\infty}^{k}(0,1)$ and $L^{\infty}=L^{\infty}(0,1)$ are denoted by 
$\|\cdot\|_{k,\infty}$ and $\|\cdot\|_{\infty}$, respectively. For a Banach space $X$,
$C([0,T];X)$ denotes as usual the continuous maps from $[0,T]$ into $X$. $\mathbb{P}_{r}$ are
the polynomials of degree at most $r$. \vspace{5pt} \\ 
\textbf{Acknowledgements}

\noindent The authors would like to thank Mr. Gregory Kounadis for the numerical experiments in the case of the ibvp (CH-s-b). V.A.D. and D.E.M. acknowledge travel support by grant MTM2014-54710 of the Ministeri\'{o} de Economia y Competividad, Spain. D.E.M. was supported by the Marsden Fund administered by the Royal Society of New Zealand with contract number VUW1418.
 
\section{Periodic splines and the quasiinterpolant}\label{sec2}   

As we will be interested in approximating solutions of the CH
that are 1-periodic functions in the spatial variable, we let $N$ be a positive integer and 
$h=1/N$, $x_{i}=ih$, $i=0,1,\ldots,N$, and for integer $r\geq 2$ consider the associated 
$N$-dimensional space of smooth 1-periodic splines
\[
S_{h} = \{\phi \in C_{per}^{r-2}[0,1] : \phi\big|_{[x_{i-1},x_{i}]} \in \mathbb{P}_{r-1}\,, 1\leq i\leq N\}\ .
\]
It is  well known that $S_{h}$ has the following approximation properties: Given a sufficiently 
smooth 1-periodic function $\rmvs$, there exists $\chi\in S_{h}$ such that
\[
\sum_{j=0}^{s-1}h^{j}\|\rmvs-\chi\|_{j} \leq C h^{s} \|\rmvs\|_{s}, \quad 1\leq s\leq r\ ,
\]
and
\[
\sum_{j=0}^{s-1}h^{j}\|\rmvs-\chi\|_{j,\infty} \leq C h^{s} \|\rmvs\|_{s,\infty}, \quad 1\leq s\leq r\ ,
\]
for some constant $C$ independent of $h$ and $\rmvs$. Moreover, there exists a constant $C$ independent of $h$ such that the inverse properties
\begin{align*}
\|\chi\|_{\beta} & \leq Ch^{-(\beta - \alpha)} \|\chi\|_{\alpha}, \quad 0\leq \alpha\leq \beta\leq r-1\ , \\
\|\chi\|_{s,\infty} & \leq Ch^{-(s +1/2)} \|\chi\|, \quad 0\leq s \leq r-1\ ,
\end{align*}
hold for all $\chi\in S_{h}$. (In the sequel we shall denote by $C$ generic constants independent of $h$.) \par
Thom\'ee and Wendroff, \cite{tw}, proved that there exists a basis $\{\phi_{j}\}_{j=1}^{N}$ of $S_{h}$ with $\mathrm{supp}(\phi_{j})=O(h)$, such that if $\rmvs$ a sufficiently smooth 1-periodic function, the associated \textit{quasiinterpolant}
$Q_{h}\rmvs =\sum_{j=1}^{N}\rmvs(x_{j})\phi_{j}$ satisfies
\begin{equation}
\|Q_{h}\rmvs - \rmvs\| \leq Ch^{r} \|\rmvs^{(r)}\|\ . 
\label{eq21}
\end{equation}
In addition, they showed that the basis $\{\phi_{j}\}_{j=1}^{N}$ may be chosen so that the following 
properties hold: \\
(i)\, If $\psi\in S_{h}$, then
\begin{equation}
\|\psi\| \leq Ch^{-1} \max_{1\leq i\leq N}|(\psi,\phi_{i})|\ .
\label{eq22}
\end{equation}
(It follows from \eqref{eq22} that if $\psi\in S_{h}$, $f\in L^{2}$ are such that 
\[
(\psi,\phi_{i}) = (f,\phi_{i}) + O(h^{\alpha}), \quad \mbox{for} \quad 1\leq i\leq N\ ,
\] 
then $\|\psi\|\leq Ch^{\alpha-1} + \|f\|$.) \\
(ii)\, Let $w$ be a sufficiently smooth 1-periodic function and $\nu$, $\kappa$ integers such that 
$0\leq\nu , \kappa\leq r-1$. Then
\begin{equation}
\bigl( (Q_{h}w)^{(\nu)},\phi_{i}^{(\kappa)}\bigr) = (-1)^{\kappa} h w^{(\nu+\kappa)}(x_{i}) + 
O(h^{2r+j-\nu-\kappa}), \quad
1\leq i\leq N\ ,
\label{eq23}
\end{equation}  
where $j=1$ if $\nu + \kappa$ is even and $j=2$ if $\nu+\kappa$ is odd. \\
(iii)\, Let $f$, $g$ be sufficiently smooth 1-periodic functions and $\nu$ and $\kappa$ as in (ii) above. Let
\[
\beta_{i} = \bigl( f(Q_{h}g)^{(\nu)}, \phi_{i}^{(\kappa)}\bigr) -(-1)^{\kappa} 
\bigl(Q_{h}\bigl[(fg^{(\nu)})^{(\kappa)}\bigr],\phi_{i}\bigr),
\quad 1\leq i\leq N\ .
\]
Then
\begin{equation}
\max_{1\leq i\leq N}|\beta_{i}| = O(h^{2r+j-\nu-\kappa})\ , 
\label{eq24}
\end{equation}
where $j$ as in (ii). \par
It follows from the approximation and inverse properties of $S_{h}$, and  
\eqref{eq21} that if $r\geq 2$, $\rmvs\in H_{per}^{r}(0,1)\cap W_{\infty}^{r}(0,1)$ and 
$V=Q_{h}\rmvs$, then %\small
\begin{align}
& \|V - \rmvs\|_{j} \leq Ch^{r-j}\|\rmvs\|_{r}, \,\, 
j=0,1,2, \,\, \text{if}\,\,
 r\geq 3, \,\, j=0,1, \,\, \text{if} \,\, r=2\ ,
\label{eq25} \\
& \|V - \rmvs\|_{j,\infty} \leq Ch^{r-j-1/2}\|\rmvs\|_{r,\infty}, \, j=0,1,2,\,
\text{if} \,\,\, r\geq 3, \, j=0,1, \, \text{if} \,\,\, r=2\ ,
\label{eq26} \\
& \|V\|_{j} \leq C, \,\, \mbox{and} \,\, \|V\|_{j,\infty} \leq C, \,\, j=0,1,2,
\,\, \text{if}\,\, r\geq 3, \,\, j=0,1, \,\, \text{if} \,\, r=2\ . 
\label{eq27}
\end{align}
\normalsize
(The inequalities \eqref{eq27} follow from \eqref{eq25} and \eqref{eq26}.)

\section{Standard Galerkin semidiscretization of the periodic problem}\label{sec3}

In this section we consider the periodic initial-value problem for the Camassa-Holm equation 
in its reduced form. For $0\leq t\leq T$ we seek $u=u(x,t)$, 1-periodic in $x$ and satisfying
\begin{equation}{\tag{C-H-per}}
\begin{aligned}
& u_{t} - u_{txx} + 3uu_{x} = 2u_{x}u_{xx} + uu_{xxx}, \quad 0\leq x\leq 1,
\quad 0\leq t\leq T\ ,\\
& u(x,0) = u_{0}(x), \quad 0\leq x\leq 1\ .
\end{aligned}
\label{eqchper}
\end{equation}
Constantin, \cite{cper}, showed that \eqref{eqchper} is well-posed in $H^{2}_{per}$ on
some temporal interval $[0,T]$, $T=T(u_{0}$). Here it will be assumed that \eqref{eqchper} 
has a unique solution that is sufficiently smooth for the purposes of the error 
estimation. \par
In the sequel we will denote the $H^{1}$ inner product by
$A(\phi,\chi) = (\phi,\chi) + (\phi',\chi')$, \,\,\, $\phi$, $\chi \in H^{1}_{per}$.
In addition we note that the right-hand side of the pde in \eqref{eqchper} may be
written as	
\[
2u_{x}u_{xx} + uu_{xxx} = u_{x}u_{xx} + (uu_{xx})_{x}= \tfrac{1}{2}(u^{2}_{x})_{x} +
(uu_{xx})_{x}\ .
\]
Taking this into account and using integration by parts we define 
the semidiscrete approximation $u_{h}$ of \eqref{eqchper} in $S_{h}$ for $r\geq 3$ 
as the map $u_{h} : [0,T]\to S_{h}$ such that
\begin{align}
& (u_{ht},\phi) + (u_{htx},\phi') + 3(u_{h}u_{hx},\phi) + 
\tfrac{1}{2}(u^{2}_{hx},\phi') + (u_{h}u_{hxx}.\phi') = 0, \quad \forall \phi \in S_{h}, 
\quad 0\leq t\leq T\ , 
\label{eq31} \\
& u_{h}(0) = Q_{h}u_{0}.
\label{eq32}
\end{align}
We start the error analysis of the semidiscretization \eqref{eq31}-\eqref{eq32} by
establishing its \textit{consistency} with \eqref{eqchper}:
\begin{lem}
Let $u$, the solution of \eqref{eqchper}, be sufficiently smooth and suppose that $r\geq 3$. If 
$U=Q_{h}u$ and $\psi :[0,T]\to S_{h}$ is such that 
\begin{equation}
A(\psi,\phi) = (U_{t},\phi) + (U_{tx},\phi') + 3(UU_{x},\phi) 
+ \tfrac{1}{2} (U_{x}^{2},\phi') + (UU_{xx},\phi'), \quad \forall \phi \in S_{h}\ ,
\label{eq33}
\end{equation}
then, there exists a constant $C$ independent of $h$, such that 
\begin{equation}
\max_{0\leq t\leq T} \|\psi(t)\|_{1} \leq Ch^{r}\ . 	
\label{eq34}
\end{equation}
\end{lem}
\begin{proof}
Let $\rho = Q_{h}u -u =U - u$. Then
\begin{equation}
A(\psi,\phi) = (\rho_{t},\phi) + (\rho_{tx},\phi') + 3(UU_{x}-uu_{x},\phi) 
+ \tfrac{1}{2}(U_{x}^{2} - u_{x}^{2},\phi') + (UU_{xx}-uu_{xx},\phi'),
\quad \forall \phi \in S_{h}\ .
\label{eq35} 	
\end{equation}
Since 
\begin{align*}
& UU_{x} - uu_{x} = (\rho + u)(\rho_{x} + u_{x}) - uu_{x} = \rho\rho_{x} 
 + u_{x}\rho + u\rho_{x}\ , \\	
& U_{x}^{2} - u_{x}^{2} = \rho_{x}(U_{x} + u_{x}) = \rho_{x}(\rho_{x} 
+ 2u_{x}) = \rho_{x}^{2} + 2u_{x}\rho_{x}\ ,\\
& UU_{xx} - uu_{xx} = (\rho + u)(\rho_{xx} + u_{xx}) - uu_{xx} = 
\rho\rho_{xx} + u_{xx}\rho + u\rho_{xx}\ ,
\end{align*}
we have from \eqref{eq35} that
\begin{equation}
A(\psi,\phi) = (\rho_{t},\phi) + 3(\rho\rho_{x}+u_{x}\rho,\phi)
+ \tfrac{1}{2}(\rho_{x}^{2},\phi') + (\rho\rho_{xx} + u_{xx}\rho,\phi') 
+ (\omega,\phi), \quad	\forall \phi \in S_{h}\ ,
\label{eq36}
\end{equation}
where $\omega : [0,T]\to S_{h}$ is defined by
\[
(\omega,\phi) = (\rho_{tx},\phi') + 3(u\rho_{x},\phi) + (u_{x}\rho_{x},\phi')
+ (u\rho_{xx},\phi'), \quad \forall \phi \in S_{h}\ .
\]
Note that for $1\leq i \leq N$,
\begin{align*}
(\omega,\phi_{i}) & = (\rho_{tx},\phi_{i}') + 3(u\rho_{x},\phi_{i})
+ (u_{x}\rho_{x},\phi_{i}') + (u\rho_{xx},\phi_{i}')\\
& = \bigl((Q_{h}u_{t})_{x} - u_{tx},\phi_{i}'\bigr) 
+ 3\bigl(u(Q_{h}u)_{x}-uu_{x},\phi_{i}\bigr)\\
& \,\,\,\,\,\, + \bigl(u_{x}(Q_{h}u)_{x} - u_{x}^{2},\phi_{i}'\bigr) 
+\bigl((u(Q_{h}u)_{xx} - uu_{xx},\phi_{i}'\bigr)\ .
\end{align*}
Therefore, from \eqref{eq24} it follows that 
\begin{align*}
(\omega,\phi_{i}) & = -(Q_{h}u_{txx}-u_{txx},\phi_{i}) 
+ 3\bigl(Q_{h}(uu_{x}) - uu_{x},\phi_{i}\bigr) \\
& \,\,\,\,\,\,\, - \bigl(Q_{h}(u_{x}^{2})_{x} - (u_{x}^{2})_{x},\phi_{i}\bigr)
- \bigl(Q_{h}(uu_{xx})_{x} - (uu_{xx})_{x},\phi_{i}\bigr) + \gamma_{i}\ ,  	
\end{align*}
where $\max_{1\leq i\leq N}|\gamma_{i}|\leq Ch^{2r-1}$. Hence, from \eqref{eq22},
\eqref{eq21} it follows that 
\begin{equation}
\|\omega\| \leq Ch^{r}\ .
\label{eq37}	
\end{equation}
Putting $\phi = \psi$ in \eqref{eq36} and taking into account \eqref{eq21}, 
\eqref{eq25}, and \eqref{eq26} we get 
\[
\|\psi\|_{1}^{2} \leq C(h^{r}\|\psi\| + h^{2r-3/2}\|\psi\| + h^{2r-5/2}\|\psi\|_{1}
+ h^{2r-2}\|\psi\|_{1}) + \|\omega\| \|\psi\|\ ,
\]
which, by \eqref{eq37} and the fact that $r\geq 3$ yields \eqref{eq34}.	\hfill $\qed$
\end{proof}
We now estimate the $L^{2}$-error of the semidiscrete scheme 
\eqref{eq31}-\eqref{eq32}. 
\begin{prop}
If $u$ is the solution (assumed to be sufficiently smooth) of \eqref{eqchper} in
$[0,1]\times[0,T]$, then for $h$ sufficiently small, the semidiscrete problem
(\ref{eq31})-(\ref{eq32}) has, for $r\geq 3$,  a  unique solution 
$u_{h}: [0,T]\to S_{h}$. Moreover, there exists a constant $C$ independent of $h$
such that 
\begin{equation}
\max_{0\leq t \leq T}\|u(t) - u_{h}(t)\| \leq Ch^{r}\ .
\label{eq38}	
\end{equation}
\end{prop}
\begin{proof}
It is clear that the ode system represented by \eqref{eq31}-\eqref{eq32} has 
a unique solution $u_{h}$ locally in $t$. While the solution exists, taking 
$\phi = u_{h}$ in (\ref{eq31}) and using integration by parts gives
\[
\tfrac{1}{2}\tfrac{d}{dt}\bigl(\|u_{h}\|^{2} + \|u_{hx}\|^{2}\bigr)
+ \tfrac{1}{2}(u_{hx}^{2},u_{hx}) + (u_{h}u_{hxx},u_{hx}) = 0\ .
\]
Due to periodicity $(u_{h}u_{hxx},u_{hx}) 
= \tfrac{1}{2}\bigl(u_{h},\partial_{x}(u_{hx}^{2})\bigr) 
= -\tfrac{1}{2}(u_{hx},u_{hx}^{2})$. Therefore, 
$\tfrac{d}{dt}\|u_{h}\|_{1}^{2} = 0$, i.e. the $H^{1}$-norm of the semidiscrete
solution is conserved. By standard ode theory this implies that $u_{h}$ exists
in any finite temporal interval. \par
Let now $U=Q_{h}u$ and $\theta = U-u_{h}$. From \eqref{eq35}, \eqref{eq31} we
have 
\begin{equation}
A(\theta_{t},\phi) + 3(UU_{x} - u_{h}u_{hx},\phi) 
+ \tfrac{1}{2}(U_{x}^{2} - u_{hx}^{2},\phi') + (UU_{xx} - u_{h}u_{hxx},\phi')
= A(\psi,\phi), \quad \forall \phi \in S_{h}\ .
\label{eq39} 	
\end{equation}
Since
\begin{align*}
& UU_{x} - u_{h}u_{hx} = UU_{x} - (U - \theta)(U_{x} - \theta_{x}) 
= (U\theta)_{x} - \theta\theta_{x}\ , \\	
& U_{x}^{2} - u_{hx}^{2} = \theta_{x}(U_{x} + u_{hx}) = \theta_{x}(2U_{x} - \theta_{x}) 
= 2U_{x}\theta_{x} - \theta_{x}^{2}\ ,\\
& UU_{xx} - u_{h}u_{hxx} = UU_{xx} - (U - \theta)(U_{xx} - \theta_{xx})
= U\theta_{xx} + U_{xx}\theta - \theta\theta_{xx}\ ,
\end{align*}
\eqref{eq39} implies for $\phi = \theta$ that 
\[
\tfrac{1}{2}\tfrac{d}{dt}\|\theta\|_{1}^{2} + 3\bigl((U\theta)_{x},\theta\bigr)
+ (U_{x}\theta_{x},\theta_{x})
- \tfrac{1}{2}(\theta_{x}^{2},\theta_{x}) + (U\theta_{xx},\theta_{x})
+ (U_{xx}\theta,\theta_{x}) - (\theta\theta_{xx},\theta_{x}) = A(\psi,\theta)\ .
\] 
Using again integration by parts we get 
\[
\tfrac{1}{2}\tfrac{d}{dt}\|\theta\|_{1}^{2} + \tfrac{3}{2}(U_{x}\theta,\theta)
+ \tfrac{1}{2}(U_{x}\theta_{x},\theta_{x}) + (U_{xx}\theta,\theta_{x})
= A(\psi,\theta)\ .
\]
Therefore, by \eqref{eq27} and \eqref{eq34} it follows that 
\begin{equation}
\tfrac{d}{dt}\|\theta\|_{1}^{2} \leq C(h^{2r} + \|\theta\|_{1}^{2}),
\quad t\in [0,T]\ .	
\label{eq310}
\end{equation}
Since $\theta(0) = 0$ by \eqref{eq32}, Gronwall's lemma gives that 
$\|\theta\|_{1} \leq Ch^{r}$ on $[0,T]$; hence \eqref{eq38} follows from \eqref{eq21}.
\hfill $\qed$
\end{proof}
\textbf{Remark} The error estimate \eqref{eq38} is still valid if $u_{h}(0)$
is taken as the $H^{1}$ (`elliptic') projection of $u_{0}$ on $S_{h}$, i.e. defined as
$u_{h}(0) = \rmvs_{h}$, where
\begin{equation}
A(\rmvs_{h},\phi) = A(u_{0},\phi), \quad \forall \phi \in S_{h}\ .
\label{eq311}
\end{equation}
Indeed, if  $\ve=Q_{h}u_{0} - \rmvs_{h}$ we have for $\phi \in S_{h}$
\begin{equation}
A(\ve,\phi) = A(Q_{h}u_{0},\phi) - A(u_{0},\phi) = (Q_{h}u_{0} - u_{0},\phi)
+ (\gamma,\phi)\ ,	
\label{eq312}
\end{equation}
where $\gamma \in S_{h}$ is defined for $\phi\in S_{h}$ by
\[
(\gamma,\phi) = \bigl((Q_{h}u_{0})' - u_{0}',\phi')\bigr)\ .
\]
Therefore, by \eqref{eq24} for $1\leq i \leq N$,
\[
(\gamma,\phi_{i}) = \bigl((Q_{h}u_{0})' - u_{0}',\phi_{i}\bigr)
= - (Q_{h}u_{0}'' - u_{0}'',\phi_{i}) + \beta_{i}\ ,
\]
where $\max_{1\leq i \leq N}|\beta_{i}| = O(h^{2r-1})$. Hence, by \eqref{eq21} 
and \eqref{eq22}, $\|\gamma\|\leq Ch^{r}$. It follows from \eqref{eq312} 
with $\phi = \ve$ that $\|\ve\|_{1}^{2} \leq Ch^{r}\|\ve\|$, i.e. that
$\|\ve\|_{1} \leq Ch^{r}$. This implies for $\theta = U - u_{h} = Q_{h}u - u_{h}$, 
that $\|\theta(0)\|_{1} = \|Q_{h}u_{0} - \rmvs_{h}\|_{1} \leq Ch^{r}$ and the
application of Gronwall's lemma to (\ref{eq310}) gives again 
$\|\theta\|_{1}\leq Ch^{r}$ on $[0,T]$, i.e. that \eqref{eq38} still holds.
Computing $\rmvs_{h}$ from \eqref{eq311} requires just the usual B-spline basis and 
not the special basis $\{\phi_{i}\}$. \hfill $\qed$ 

\section{A modified Galerkin method for the periodic problem in system form}\label{sec4}

In this section we consider a \textit{modified} Galerkin method for the periodic ivp 
for CH, that works also if $r=2$, i.e. when the finite element space consists of 1-periodic,
continuous, piecewise linear functions on a uniform mesh in $[0,1]$, consequently
being a subspace of $H^{1}_{per}$. \par
For this purpose, following \cite{cper}, we write \eqref{eqchper} in system form for
two 1-periodic functions $m$ and $u$ as follows:
\begin{equation}{\tag{C-H-s-per}}
\begin{aligned}
\begin{aligned}
& m = u - u_{xx}\ , \\
& m_{t} + (mu)_{x} + mu_{x} = 0\ ,	
\end{aligned}
\,\,\,\,\, & \,\,\, 0\leq x\leq 1, \,\,\, 0\leq t\leq T,	\\
u(x,0) = u_{0}(x), \quad m(x,0) & = u_{0}(x) - u_{0}''(x), 
\quad 0\leq x\leq 1\ .
\end{aligned}
\label{eqchsper}
\end{equation}
We will discretize \eqref{eqchsper} in space using approximations $u_{h}$, $m_{h}$ of
$u$, $m$ that take values for $0\leq t\leq T$ in $S_{h}$ for $r\geq 2$. This will
yield a \textit{modified} Galerkin method defined for $0\leq t\leq T$ by 
\begin{align}
& (m_{h},\phi) = (u_{h},\phi) + (u_{hx},\phi'), \quad \forall \phi \in S_{h}\ ,	
\label{eq41} \\
& (m_{ht},\phi) + \bigl((m_{h}u_{h})_{x},\phi\bigr) + (m_{h}u_{hx},\phi)=0,
\quad \forall \phi \in S_{h}\ ,
\label{eq42}
\end{align}
with initial value
\begin{equation}
(m_{h}(0),\phi) = (u_{0},\phi) + (u_{0}',\phi'), \quad \phi \in S_{h}\ ,
\label{eq43}	
\end{equation}
i.e. with $m_{h}(x,0)$ taken as the $L^{2}$-projection of $m(x,0)$ on $S_{h}$.
Compatibility with (\ref{eq41}) and (\ref{eq43}) imply that $u_{h}(0)$ satisfies
for $\phi \in S_{h}$
\begin{equation}
A(u_{h}(0),\phi) = A(u_{0},\phi)\ ,
\label{eq44}	
\end{equation}
i.e. that $u_{h}(0)$ is the $H^{1}$ projection of $u_{0}$ in $S_{h}$. We first
establish the {\em{consistency}} of \eqref{eq41}-\eqref{eq43} with \eqref{eqchsper}.
\begin{lem}
Let $(m,u)$, the solution of \eqref{eqchsper} in $[0,1]\times[0,T]$, be sufficiently
smooth and suppose that $r\geq 2$. 
If $M = Q_{h}m$, $U = Q_{h}u$ and  $\psi$, $\zeta :[0.T]\to S_{h}$ are such that 
\begin{equation}
\begin{aligned}
& (U,\phi) + (U_{x},\phi') - (M,\phi) = (\psi,\phi), \quad \forall \phi \in S_{h},\\
& (M_{t},\phi) + \bigl((MU)_{x},\phi\bigr) + (MU_{x},\phi) = (\zeta,\phi)
\quad \forall \phi \in S_{h}\ ,
\end{aligned}
\label{eq45}	
\end{equation}
then, there exists a constant $C$ independent of $h$, such that 
\begin{equation}
\max_{0\leq t\leq T}(\|\psi(t)\| + \|\zeta(t)\|) \leq Ch^{r}\ .
\label{eq46}	
\end{equation}
\end{lem}
\begin{proof}
Let $\rho = Q_{h}u - u = U - u$, and $\sigma = Q_{h}m - m = M - m$. Then
\begin{equation}
\begin{aligned}
& (\psi,\phi) = (\rho,\phi) + (\rho_{x},\phi') - (\sigma,\phi) 
= (\rho - \sigma,\phi) + (\wt{\psi},\phi), \quad \forall \phi \in S_{h}\ ,\\
& (\zeta,\phi) = (\sigma_{t},\phi) + \bigl((MU)_{x}- (mu)_{x},\phi\bigr)
+ (MU_{x} - mu_{x},\phi), \quad \forall \phi \in S_{h}\ ,
\end{aligned}
\label{eq47}	
\end{equation}
where $\wt{\psi} : [0,T] \to S_{h}$ satisfies
\begin{equation}
(\wt{\psi},\phi) = (\rho_{x},\phi'), \quad \forall \phi \in S_{h}\ .
\label{eq48}	
\end{equation}
Since
\begin{align*}
& MU - mu = (m + \sigma)(u + \rho) - mu = m\rho + u\sigma + \sigma\rho\ ,\\
& MU_{x} - mu_{x} = (m + \sigma)(u_{x} + \rho_{x}) - mu_{x} = m\rho_{x} 
+ u_{x}\sigma + \sigma\rho_{x}\ ,	
\end{align*}
it follows from the second equation of \eqref{eq47} that 
\begin{equation}
(\zeta,\phi) = (\sigma_{t},\phi) + (m_{x}\rho,\phi) + 2(u_{x}\sigma,\phi)
+ (\sigma_{x}\rho,\phi) + 2(\sigma\rho_{x},\phi) + (\wt{\zeta},\phi), 
\quad\forall \phi \in S_{h}\ ,
\label{eq49}	
\end{equation}
where $\wt{\zeta} : [0,T] \to S_{h}$ is defined as
\begin{equation}
(\wt{\zeta},\phi) = 2(m\rho_{x},\phi) + (u\sigma_{x},\phi), \quad
\forall \phi \in S_{h}\ .
\label{eq410}	
\end{equation}
It follows from \eqref{eq48} for $1\leq i\leq N$, 
\[
(\wt{\psi},\phi_{i}) = (\rho_{x},\phi_{i}') = \bigl((Q_{h}u)_{x},\phi_{i}'\bigr)
 - (u_{x},\phi_{i}') = - (Q_{h}u_{xx} - u_{xx},\phi_{i}) + \gamma_{i}\ ,
\]
where, from \eqref{eq24}, $\max_{1\leq i\leq N}|\gamma_{i}| \leq Ch^{2r-1}$.
Therefore, from \eqref{eq22} and \eqref{eq21} we see that 
\begin{equation}
\|\wt{\psi}\| \leq Ch^{r}\ .
\label{eq411}
\end{equation} 	
Putting $\phi = \psi$ in the first equation of \eqref{eq47} and taking into account
\eqref{eq21} and \eqref{eq411} we get 
\begin{equation}
\|\psi\| \leq Ch^{r}\ .
\label{eq412}	
\end{equation}
In addition, from \eqref{eq410} it follows for $1 \leq i\leq N$
\begin{align*}
(\wt{\zeta},\phi_{i}) & = 2\bigl(m(Q_{h}u)_{x},\phi_{i}\bigr) 
- 2(mu_{x},\phi_{i}) + \bigl(u(Q_{h}m)_{x},\phi_{i}\bigr) - (um_{x},\phi_{i})\\
& = 2\bigl(Q_{h}(mu_{x}) - mu_{x},\phi_{i}\bigr) 
+ \bigl(Q_{h}(um_{x}) - um_{x},\phi_{i}\bigr) + \wt{\gamma}_{i}\ ,	
\end{align*}
where, by \eqref{eq24}, $\max_{1\leq i\leq N}|\wt{\gamma}_{i}| \leq Ch^{2r+1}$.
Consequently, from \eqref{eq22} and \eqref{eq21} we have
\begin{equation}
\|\wt{\zeta}\| \leq Ch^{r}\ .
\label{eq413}	
\end{equation}
Finally, if we take $\phi = \zeta$ in \eqref{eq49}, we obtain, in view of \eqref{eq413},
\[
\|\zeta\| \leq Ch^{r}\ ,
\]
which, together with \eqref{eq412}, establishes \eqref{eq46}. \hfill $\qed$
\end{proof}
We proceed now to derive an $L^{2}$-error estimate for the solution of the
semidiscrete problem \eqref{eq41}-\eqref{eq43}.
\begin{prop}
If the solution $(m,u)$ of \eqref{eqchsper} on $[0,1]\times[0,T]$ is sufficiently
smooth, then, for $h$ sufficiently small, the semidiscrete problem 
\eqref{eq41}-\eqref{eq43} has, for $r\geq 2$, a unique solution on $[0,T]$ that satisfies
\begin{equation}
\max_{0\leq t\leq T}(\|m(t) - m_{h}(t)\| + \|u(t) - u_{h}(t)\|) \leq Ch^{r}\ .
\label{eq414}	
\end{equation}
\end{prop}
\begin{proof}
For every $t\geq 0$ we may solve the linear discrete problem \eqref{eq41} and
express $u_{h}$ in terms of $m_{h}$. Upon substituting $u_{h}$ in \eqref{eq42}, we
easily see that the resulting initial-value problem \eqref{eq42}-\eqref{eq43}
has a unique solution $m_{h}$ locally in $t$. While this solution exists, putting
$M = Q_{h}m$, $U=Q_{h}u$, $\theta = U - u_{h}$, and $\xi = M - m_{h}$, from
\eqref{eq45}, \eqref{eq41}, \eqref{eq42} we obtain 
\begin{equation}
\begin{aligned}
& (\theta,\phi) + (\theta_{x},\phi') - (\xi,\phi) = (\psi,\phi), 
\quad \forall \phi \in S_{h}\ ,\\
& (\xi_{t},\phi) + \bigl((MU - m_{h}u_{h})_{x}, \phi\bigr) + (MU_{x} - m_{h}u_{hx},\phi) 
= (\zeta,\phi), \quad \forall \phi \in S_{h}\ .
\end{aligned}
\label{eq415}	
\end{equation}
With $\phi = \theta$ in the first equation above we see, in view of \eqref{eq46},
that
\begin{equation}
\|\theta\|_{1} \leq \|\xi\| + Ch^{r}\ .
\label{eq416}	
\end{equation}
Now
\begin{align*}
& MU - m_{h}u_{h} = MU - (M - \xi)(U - \theta) = M\theta + U\xi - \theta\xi\ ,\\
& MU_{x} - m_{h}u_{hx} = MU_{x} - (M - \xi)(U_{x} - \theta_{x}) 
= M\theta_{x} + U_{x}\xi - \theta_{x}\xi\ .	
\end{align*}
Hence,  from the second equation in \eqref{eq415}, putting $\phi = \xi$ we obtain
\[
\tfrac{1}{2}\tfrac{d}{dt}\|\xi\|^{2} + \bigl((M\theta)_{x},\xi\bigr)
+ \bigl((U\xi)_{x},\xi\bigr) - \bigl((\theta\xi)_{x},\xi\bigr) 
+ (M\theta_{x},\xi) + (U_{x}\xi,\xi) - (\theta_{x}\xi,\xi) = (\zeta,\xi)\ .
\]
Using integration by parts we see that \small
$\bigl((U\xi)_{x},\xi\bigr) = - (U\xi,\xi_{x}) = \tfrac{1}{2}(U_{x}\xi,\xi)$, 
\normalsize and
similarly, \small $\bigl((\theta\xi)_{x},\xi\bigr) = \tfrac{1}{2}(\theta_{x}\xi,\xi)$.
\normalsize
Therefore, we may write the above equation as
\[
\tfrac{1}{2}\tfrac{d}{dt}\|\xi\|^{2} + (M_{x}\theta,\xi) + 2(M\theta_{x},\xi)
+ \tfrac{3}{2}(U_{x}\xi,\xi) - \tfrac{3}{2}(\theta_{x}\xi,\xi) = (\zeta,\xi)\ .
\]  
Consequently, from \eqref{eq27}, \eqref{eq46}, \eqref{eq416}, we get for some
constant $C$ independent of $h$, that
\begin{equation}
\tfrac{d}{dt}\|\xi\|^{2} \leq C(\|\xi\|^{2} + h^{r}\|\xi\|) 
+ 3|(\theta_{x}\xi,\xi)|\ .
\label{eq417}	
\end{equation}
The definition of $m_{h}(0)$ by \eqref{eq43} implies that for
$\xi(0) = Q_{h}m(0) - m_{h}(0)$ we have if $1\leq i\leq N$
\begin{align*}
\bigl(\xi(0),\phi_{i}\bigr) & = \bigl(Q_{h}m(0)-m_{h}(0),\phi_{i}\bigr)  
= \bigl(Q_{h}(u_{0} - u_{0}''),\phi_{i}\bigr) - (u_{0},\phi_{i})
- (u_{0}',\phi_{i}') \\
& = (Q_{h}u_{0} - u_{0},\phi_{i}) - (Q_{h}u_{0}'' - u_{0}'',\phi_{i})\ .
\end{align*}
Therefore, by \eqref{eq22} and \eqref{eq21} we conclude that
\begin{equation}
\|\xi(0)\| \leq C h^{r}\ ,
\label{eq418}	
\end{equation}
from which, by inverse properties, we see that $\|\xi(0)\|_{\infty} = O(h^{r-1/2})$. By
continuity we may infer that there is a maximal time $t_{h}\in (0,T]$ such that
the solution of \eqref{eq41}-\eqref{eq43} exists for $0\leq t\leq t_{h}$ and
satisfies
\begin{equation}
\max_{0\leq t\leq t_{h}}\|\xi(t)\|_{\infty} \leq 1\ .
\label{eq419}	
\end{equation}
Therefore, from \eqref{eq417} it follows that
\[
\tfrac{d}{dt}\|\xi\|^{2} \leq C(\|\xi\|^{2} + h^{r}\|\xi\|) 
+ 3\|\theta\|_{1}\|\xi\| \leq C(h^{2r} + \|\xi\|^{2})\ ,
\]
for $t\leq t_{h}$. From this inequality and Gronwall's lemma we infer that
\[
\|\xi(t)\| \leq C_{T} (\|\xi(0)\| + h^{r})\ ,
\]
for $t\leq t_{h}$. Hence, by \eqref{eq418} 
\begin{equation}
\|\xi(t)\| \leq C_{T}h^{r}, \quad 0\leq t\leq t_{h}\ ,
\label{eq420}	
\end{equation}
from which, by inverse properties we conclude that
\[
\|\xi(t)\|_{\infty} \leq C_{T}h^{r-1/2}\ ,
\]
for $t\leq t_{h}$. Therefore, if $h$ is taken sufficiently small, $t_{h}$ is not
maximal and one can consequently take $t_{h}=T$. Hence, \eqref{eq420} holds
for $t\in [0,T]$. By \eqref{eq416}, $\|\theta\|_{1}\leq Ch^{r}$ on $[0,T]$
and \eqref{eq414} follows in view of \eqref{eq21}. 	\hfill $\qed$
\end{proof}

\section{A Galerkin method for the initial-boundary-value problem}\label{sec5}

In this  section we consider the following initial-boundary value problem (ibvp) for the 
Camassa-Holm equation. For $0\leq t\leq T$ and $0\leq x\leq 1$ we seek $u=u(x,t)$ satisfying
\begin{equation}{\tag{C-H-b}}
\begin{aligned}
u_{t}-u_{txx} + 3uu_{x} & = 2u_{x}u_{xx} +uu_{xxx}, \,\,\,\,\, 0\leq x\leq 1, \,\,\,\,\, 
0\leq t\leq T\ , \\
u(x,0) &  = u_{0}(x), \,\,\,\,\, 0\leq x\leq 1\ ,\\	
u(0,t) = u_{xx}(0,t) & = u(1,t) = u_{xx}(1,t) = 0, \,\,\,\,\, 0\leq t\leq T\ .
\end{aligned}
\label{eqchb}
\end{equation}
This ibvp was first analyzed by Kwek \textit{et al.} \cite{k}, who showed that there exists $T>0$ 
such that it has a unique solution which is a continuous map of $[0,T]$ into $\mathcal{B}$,
where 
$\mathcal{B} = \{\rmvs\in H^{4}(0,1): \rmvs(0)=\rmvs''(0)=\rmvs(1)=\rmvs''(1)=0\}$. 
The form of the boundary
conditions suggests that the problem should be naturally studied in its $m$-$u$ system form,
and this was actually done in \cite{k}. Consequently, we seek $u=u(x,t)$ and $m=m(x,t)$,
defined for $(x,t) \in [0,1]\times[0,T]$ and satisfying 
\begin{equation}{\tag{C-H-s-b}}
\begin{aligned}
\begin{aligned}
& m = u - u_{xx}\ , \\
& m_{t} + um_{x} + 2u_{x}m = 0\ ,	
\end{aligned}
\,\,\,\,\, & \,\,\, 0\leq x\leq 1, \,\,\, 0\leq t\leq T\ ,	\\
m(0,t) = m(1,t) = u(0,&\, t)  =u(1,t) = 0, \,\,\, 0\leq t\leq T\ , \\
u(x,0) = u_{0}(x), \,\,\, m(x,0) & = m_{0}(x), \,\,\, 0\leq x\leq 1\ ,
\end{aligned}
\label{eqchsb}
\end{equation}
where $m_{0} = u_{0} - u_{0}''$. \par
In the sequel we will analyze a Galerkin-finite element method for the ibvp \eqref{eqchsb} and
also make some remarks on the numerical solution of the ibvp \eqref{eqchb} for the single equation
by a standard Galerkin method. We assume that these ibvp's possess solutions that are smooth
enough for the purposes of the error estimation. \par
Letting $A(\cdot,\cdot)$ be the bilinear form considered in Section \ref{sec3}, but now defined for
$\rmvs$, $w \in \mathring{H}^{1}$ by 
\begin{equation}
A(\rmvs,w) = (\rmvs,w) + (\rmvs',w')\ ,
\label{eq51}
\end{equation}
i.e. as the $H^{1}$ inner product, we may write \eqref{eqchsb} in weak form for $0\leq t\leq T$ as
\begin{equation}
(m,\rmvs) = A(u,\rmvs), \,\,\,\, \forall \rmvs\in \mathring{H}^{1}\ ,
\label{eq52}
\end{equation}
and
\begin{equation}
(m_{t},\rmvs) = - (um_{x},\rmvs) - (2u_{x}m,\rmvs), \,\,\,\, \forall \rmvs \in \mathring{H}^{1}\ .
\label{eq53}
\end{equation}
In order to define the numerical method, let $0=x_{1}<x_{2}<\dots<x_{N+1}=1$ be a quasiuniform
partition of $[0,1]$ with maximum meshlength $h$. For integers $r\geq 2$ and $0\leq\mu\leq r-2$ 
let $S_{h}^{r,\mu} = \{\phi \in C^{\mu}[0,1]: \phi |_{[x_{i},x_{i+1}]} \in \mathbb{P}_{r-1},
\, i=1,2,\dots,N\}$ and $S_{h} = S_{h}^{r,\mu} \cap \mathring{H}^{1}$. By 
$R_{h} : \mathring{H}^{1} \to S_{h}$ denote the elliptic projection onto $S_{h}$ defined as
usual for $\rmvs \in \mathring{H}^{1}$ as $A(R_{h}\rmvs,\phi) = A(\rmvs,\phi)$ for all 
$\phi \in S_{h}$ and
let $P_{h} : L^{2} \to S_{h}$ be the $L^{2}$ projection onto $S_{h}$. From well-known approximation
properties of the finite element spaces $S_{h}$, \cite{df}, \cite{s}, it follows that
\begin{align}
\|R_{h}\rmvs - \rmvs\|_{j} & \leq Ch^{k-j}\|\rmvs\|_{k},\,\,\,\, j=0,1,\,\,\, 2\leq k\leq r,\,\,\, 
\rmvs \in H^{k}\cap H_{0}^{1}\ , 
\label{eq54} \\
\|P_{h}\rmvs - \rmvs\|_{j} & \leq Ch^{k-j}\|\rmvs\|_{k}, \,\,\, j=0,1, \,\,\, 2\leq k\leq r,\,\,\,
\rmvs \in H^{k}\cap H_{0}^{1}\ .
\label{eq55}
\end{align} 
We may also asume the following stability properties of $R_{h}$ and $P_{h}$, valid for functions
in the spaces implied by the indicated norms, that also vanish at $x=0$ and $x=1$.
\begin{align}
\|R_{h}\rmvs\|_{\infty} & \leq C\|\rmvs\|_{\infty}, \quad \|R_{h}\rmvs\|_{1,\infty} \leq C\|\rmvs\|_{1,\infty}\ ,
\label{eq56} \\
\|P_{h}\rmvs\|_{1} & \leq C\|\rmvs\|_{1}, \quad \, \|P_{h}\rmvs\|_{\infty} \leq C\|\rmvs\|_{\infty}\ .
\label{eq57}
\end{align}
(For the stability of $R_{h}$ in the $\|\cdot\|_{\infty}$ norm and of $P_{h}$ in $\mathring{H}^{1}$
cf. \cite{twa}. For the stability of $P_{h}$ in $L^{\infty}$ cf. \cite{ddw}, while for that of
$R_{h}$ in the $\|\cdot\|_{1,\infty}$ norm, cf. \cite{rs}, \cite{adm}.) 
In addition, note that the inverse inequalities $\|\chi\|_{1}\leq Ch^{-1}\|\chi\|$,
$\|\chi\|_{j,\infty}\leq Ch^{-(j+1/2)}\|\chi\|$, $j=0,1$, hold for $\chi \in S_{h}$. The 
semidiscrete approximation of \eqref{eqchsb} is now defined as follows. We seek
$m_{h}$, $u_{h} : [0,T] \to S_{h}$ satisfying for $0\leq t\leq T$
\begin{align}
&(m_{h},\phi) = A(u_{h},\phi), \quad \forall\phi \in S_{h}\ ,
\label{eq58}\\
&(m_{ht},\chi) = -(u_{h}m_{hx},\chi) - 2(u_{hx}m_{h},\chi), \quad \forall \chi \in S_{h}\ ,
\label{eq59}
\end{align}
with initial conditions
\begin{equation}
m_{h}(0) = P_{h}(u_{0} - u_{0}''), \quad u_{h}(0) = R_{h}u_{0}\ .
\label{eq510}
\end{equation}
\begin{prop}\label{prop51} Let $r\geq 2$ and suppose that the ibvp \eqref{eqchsb} has a unique solution
$(m,u)$, with $m\in C(0,T;C^{r-2} \cap \mathring{H}^{1})$, 
$u\in C(0,T;C^{r}\cap\mathring{H}^{1})$.
Then, if $h$ is sufficiently small, the semidiscrete problem \eqref{eq58}-\eqref{eq510} has a
unique solution $(m_{h}, u_{h})$ in $[0,T]$ for which 
\begin{equation}
\max_{0\leq t\leq T}(\|m(t) - m_{h}(t)\| + \|u(t) - u_{h}(t)\|_{1}) \leq Ch^{r-1}\ .
\label{eq511}
\end{equation}
\end{prop} 
\begin{proof} Let $\rho = u - R_{h}u$, $\theta = R_{h}u-u_{h}$, $\sigma=m - P_{h}m$, and
$\xi = P_{h}m - m_{h}$. Then $u -u_{h} = \rho + \theta$, $m - m_{h} = \sigma + \xi$, and it
suffices to show that for $0\leq t\leq T$
\[
\|\xi\| + \|\theta\|_{1} \leq Ch^{r-1}\ .
\]
It is clear that the initial value o.d.e. problem represented by \eqref{eq58}-\eqref{eq510} has a
unique solution $(m_{h}, u_{h})$ locally in $t$. While this solution exists, subtracting
\eqref{eq58} from \eqref{eq52} (with $\rmvs = \phi$) and \eqref{eq59} from \eqref{eq53} 
(with $w=\chi$) gives
\begin{align}
& (\xi,\phi) = A(\theta,\phi), \quad \forall \phi \in S_{h}\ ,
\label{eq512} \\
& (\xi_{t},\chi) = -(um_{x} - u_{h}m_{hx},\chi) - 2(u_{x}m - u_{hx}m_{h},\chi), \quad
\forall \chi \in S_{h}\ .
\label{eq513}
\end{align}
Since
\begin{align}
um_{x} - u_{h}m_{hx} & = (m_{x}-m_{hx})u + (u - u_{h})m_{x} - (u-u_{h})(m_{x}-m_{hx})\ ,
\label{eq514} \\
u_{x}m - u_{hx}m_{h} & = (m - m_{h})u_{x} + (u_{x} - u_{hx})m - (u_{x}-u_{hx})(m - m_{h})\ ,
\label{eq515}
\end{align} 
it follows that 
\begin{align*}
um_{x} - u_{h}m_{hx} & = \sigma_{x}u + \xi_{x}u +\rho m_{x} + \theta m_{x} 
- \rho\sigma_{x} - \rho\xi_{x} - \theta\sigma_{x} - \theta\xi_{x}\ ,\\
u_{x}m - u_{hx}m_{h} & = \sigma u_{x} + \xi u_{x} + \rho_{x}m + \theta_{x}m
- \rho_{x}\sigma - \rho_{x}\xi - \theta_{x}\sigma - \theta_{x}\xi\ ,
\end{align*}
and therefore
\begin{align*}
(\xi_{t},\chi) = & - (\sigma_{x}u+2\sigma u_{x},\chi) - (\xi_{x}u+2\xi u_{x},\chi)
-(\rho m_{x}+2\rho_{x}m,\chi) - (\theta m_{x} + 2\theta_{x}m,\chi) \\
& + (\rho\sigma_{x}+2\rho_{x}\sigma,\chi) + (\rho\xi_{x} + 2\rho_{x}\xi,\chi)
+(\theta\sigma_{x}+2\theta_{x}\sigma,\chi) + (\theta\xi_{x} + 2\theta_{x}\xi,\chi),
\,\,\, \forall \chi \in S_{h}\ .
\end{align*}
Putting $\chi=\xi$ in this relation we get
\begin{equation}
(\xi_{t},\xi) = - w_{1} - w_{2} - w_{3} - w_{4} + w_{5} + w_{6} + w_{7} + w_{8}\ ,
\label{eq516}
\end{equation}
where
\begin{align*}
w_{1} & = (\sigma_{x}u + 2\sigma u_{x},\xi), \,\,\, w_{2} = \tfrac{3}{2}(\xi u_{x},\xi),
\,\,\, w_{3} = (\rho m_{x} + 2\rho_{x}m,\xi), \,\,\, w_{4} = (\theta m_{x} + 2\theta_{x}m,\xi)\ ,\\
w_{5} & = (\rho\sigma_{x} + 2\rho_{x}\sigma,\xi)\,\,\, w_{6} = \tfrac{3}{2}(\rho_{x}\xi,\xi),\,\,\,
w_{7} = (\theta\sigma_{x} + 2\theta_{x}\sigma,\xi),\,\,\, w_{8} = \tfrac{3}{2}(\theta_{x}\xi,\xi)\ .
\end{align*}
Taking into account \eqref{eq54} and \eqref{eq55} we may estimate $w_{1}$, $w_{2}$, $w_{3}$ as
follows:
\begin{equation}
\begin{aligned}
\abs{w_{1}} & \leq (\|\sigma_{x}\|\|u\|_{\infty} + 2\|\sigma\|\|u_{x}\|_{\infty})\|\xi\| 
\leq Ch^{r-1}\|\xi\|\ , \\
\abs{w_{2}} & \leq \tfrac{3}{2}\|u_{x}\|_{\infty}\|\xi\|^{2} \leq C\|\xi\|^{2}\ , \\
\abs{w_{3}} & \leq (\|\rho\|\|m_{x}\|_{\infty} + 2\|\rho_{x}\|\|m\|_{\infty})\|\xi\|
\leq Ch^{r-1}\|\xi\|\ .
\end{aligned}
\label{eq517}
\end{equation}
To estimate $w_{4}$ we take into account \eqref{eq512}, which implies that 
\begin{equation}
\|\theta\|_{1} \leq \|\xi\|\ .
\label{eq518}
\end{equation}
Hence 
\begin{equation}
\abs{w_{4}} \leq (\|m_{x}\|_{\infty}\|\theta\| + 2\|m\|_{\infty}\|\theta_{x}\|)\|\xi\|
\leq C\|\xi\|^{2}\ .
\label{eq519}
\end{equation}
For $w_{5}$ and $w_{6}$, using the stability of $R_{h}$ in $L^{\infty}$ and $W_{\infty}^{1}$,
cf. \eqref{eq56}, and \eqref{eq55} gives
\begin{align}
\abs{w_{5}} & \leq (\|\rho\|_{\infty} \|\sigma_{x}\| + 2\|\rho_{x}\|_{\infty}\|\sigma\|)\|\xi\|
\leq C h^{r-1} \|\xi\|\ , 
\label{eq520} \\
\abs{w_{6}} & \leq \tfrac{3}{2}\|\rho_{x}\|_{\infty}\|\xi\|^{2} \leq C\|\xi\|^{2}\ .
\label{eq521}
\end{align}
To estimate $w_{7}$, by Sobolev's inequality, \eqref{eq518}, and the stability of $P_{h}$
in $\mathring{H}^{1}$ and $L^{\infty}$, cf. \eqref{eq57}, we get
\begin{equation}
\abs{w_{7}} \leq (\|\theta\|_{\infty}\|\sigma_{x}\| + 2\|\theta_{x}\|\|\sigma\|_{\infty})\|\xi\|
\leq C\|\xi\|^{2}\ .
\label{eq522}
\end{equation}
In order to estimate $w_{8}$, we let $0<t_{h}<T$ be the maximal time such that the solution of
\eqref{eq58}-\eqref{eq510} exists and satisfies 
$\max_{0\leq t\leq t_{h}}\|\xi(t)\|_{\infty}\leq 1$. Then, by \eqref{eq518}
\begin{equation}
\abs{w_{8}} \leq \tfrac{3}{2}\|\theta_{x}\|\|\xi\|_{\infty} \|\xi\| \leq C\|\xi\|^{2}\ .
\label{eq523}
\end{equation}
In view of \eqref{eq517}-\eqref{eq523}, \eqref{eq516} gives for $0\leq t\leq t_{h}$
\[
\tfrac{d}{dt}\|\xi(t)\|^{2} \leq C(h^{2r-2} + \|\xi(t)\|^{2})\ .
\]
Hence Gronwall's lemma implies, for some constant $C=C(T)$, that
\begin{equation}
\|\xi(t)\| \leq Ch^{r-1}, \quad 0\leq t\leq t_{h}\ .
\label{eq524}
\end{equation} 
Since $\|\xi(t)\|_{\infty} \leq Ch^{-1/2}\|\xi\|\leq Ch^{r-3/2}$ from \eqref{eq524}, it follows
that for $h$ small enough $\|\xi\|_{\infty} < 1$ that implies $t_{h}$ was not maximal and that
the argument may be repeated up to $t=T$. It follows that \eqref{eq524} holds on $[0,T]$,
and by \eqref{eq518}, we get that $\|\xi\| + \|\theta\|_{1} = O(h^{r-1})$ on $[0,T]$, proving
\eqref{eq511}. \hfill $\qed$
\end{proof}
The rate of convergence $r-1$ in Proposition \ref{prop51} is optimal for the error $u-u_{h}$ in the
$H^{1}$ norm. In the $L^{2}$ norm we prove in the following proposition, using a duality
argument, a suboptimal convergence rate estimate.
\begin{prop} Under the hypotheses of Proposition $5.1$ there holds
\begin{equation}
\max_{0\leq t\leq T}\|u(t) - u_{h}(t)\|\leq Ch^{r-1/2}\ . 
\label{eq525}
\end{equation}
\end{prop}  
\begin{proof}
We use the notation of the proof of Proposition \ref{prop51} If $\rmvs$ is the solution of the problem
\begin{equation}
\begin{aligned}
 \rmvs - \rmvs_{xx} & = 0\ , \\
\rmvs(0,t) = & \rmvs(1,t) = 0\ ,
\end{aligned}
\label{eq526}
\end{equation}
then, taking into account \eqref{eq512} and \eqref{eq513}, it follows that 
\begin{equation}
(\theta_{t},\theta) = A(\theta_{t},\rmvs) = A(\theta_{t},R_{h}\rmvs) = (\xi_{t},R_{h}\rmvs)
= -(e,R_{h}\rmvs - \rmvs) - (e,\rmvs)\ ,
\label{eq527}
\end{equation}
where $e = e_{1} + 2e_{2}$, and 
\begin{align*}
e_{1} & = (m_{x} - m_{hx})u + (u - u_{h})m_{x} - (u -u _{h})(m_{x} - m_{hx})\ , \\
e_{2} & = (m - m_{h})u_{x} + (u_{x} - u_{hx})m - (u_{x} - u_{hx})(m -m_{h})\ .
\end{align*}
(The last inequality of \eqref{eq527} follows from noting that 
$e = e_{1} + 2e_{2} = um_{x} - u_{h}m_{hx} + 2mu_{x} - 2m_{h}u_{hx}$, and \eqref{eq513}
with $\chi = R_{h}\rmvs$.) We now estimate the last two terms in \eqref{eq527}. In view of 
\eqref{eq54} and \eqref{eq526} we have
\begin{equation}
\abs{(e,R_{h}\rmvs - \rmvs)} \leq \|e\| \|R_{h}\rmvs - \rmvs\|\leq Ch^{2}\|\rmvs\|_{2} \|e\|
\leq Ch^{2}\|\theta\| \|e\|\ .
\label{eq528}
\end{equation}
Using \eqref{eq511}, Sobolev's inequality, and e.g. \eqref{eq55} and the inverse inequalities
in $S_{h}$ we see that 
\begin{align*}
\|e_{1}\| & \leq \|u\|_{\infty} \|m_{x} - m_{hx}\| + \|m_{x}\|_{\infty} \|u -u _{h}\|
+ \|u -u_{h}\|_{\infty} \|m_{x} - m_{hx}\|\\
& \leq C(h^{r -2} + h^{r-1} + h^{2r-3}) \leq Ch^{r-2}\ .
\end{align*} 
Similarly $\|e_{2}\| \leq Ch^{r-2}$, and \eqref{eq528} gives
\begin{equation}
\abs{(e,R_{h}\rmvs - \rmvs)} \leq Ch^{r}\|\theta\|\ .
\label{eq529}
\end{equation}
The definition of $e$ gives 
\begin{align*}
(e,\rmvs) & = (\sigma_{x}u + 2\sigma u_{x},\rmvs) + (\rho m_{x} + 2\rho_{x}m,\rmvs)
+ (\theta m_{x} + 2\theta_{x}m,\rmvs) - (\rho\sigma_{x} + 2\rho_{x}\sigma,\rmvs) \\
& \,\,\,\,\,\,\,\, -(\rho\xi_{x} + 2\rho_{x}\xi,\rmvs) - (\theta\sigma_{x} 
+ 2\theta_{x}\sigma,\rmvs)
- (\theta\xi_{x} + 2\theta_{x}\xi,\rmvs) + (\xi_{x}u+2\xi u_{x},\rmvs)\ .
\end{align*}
Hence, integrating by parts,
\begin{equation}
\begin{aligned}
(e,\rmvs) & = (\sigma,u_{x}\rmvs - u\rmvs_{x}) - (\rho,m_{x}\rmvs + 2m\rmvs_{x}) 
- (\theta,m_{x}\rmvs + 2m\rmvs_{x})  - (\rho\sigma_{x} + 2\rho_{x}\sigma,\rmvs)\\
&\,\,\,\,\,\,\,\, + (\xi,\rho \rmvs_{x} - \rho_{x}\rmvs) 
+ (\sigma,\theta \rmvs_{x} - \theta_{x}\rmvs)
+(\xi,\theta \rmvs_{x} - \theta_{x}\rmvs) + (\xi,u_{x}\rmvs - u\rmvs_{x}) \\
& = : \wt{w}_{1} + \wt{w}_{2}\ ,
\end{aligned}
\label{eq530}
\end{equation}
where 
\begin{align*}
\wt{w}_{1} & = (\sigma,u_{x}\rmvs - u\rmvs_{x}) - (\rho,m_{x}\rmvs + 2m\rmvs_{x}) 
- (\theta,m_{x}\rmvs + 2m\rmvs_{x}) - (\rho\sigma_{x} + 2\rho_{x}\sigma,\rmvs)\\
&\,\,\,\,\,\,\,\, + (\xi,\rho \rmvs_{x} - \rho_{x}\rmvs) 
+ (\sigma,\theta \rmvs_{x} - \theta_{x}\rmvs) +(\xi,\theta \rmvs_{x} - \theta_{x}\rmvs)\ ,
\end{align*}
and 
\[
\wt{w}_{2} = (\xi,u_{x}\rmvs - u\rmvs_{x})\ .
\]
To estimate $\wt{w}_{1}$ note that
\begin{align*}
\abs{\wt{w}_{1}} & \leq \|\sigma\|(\|u_{x}\|_{\infty}\|\rmvs\| + \|u\|_{\infty}\|\rmvs_{x}\|)
+ \|\rho\|(\|m_{x}\|_{\infty} \|\rmvs\| + 2\|m\|_{\infty}\|\rmvs_{x}\|) \\
&\,\,\,\,\,\,\,\, + \|\theta\|(\|m_{x}\|_{\infty}\|\rmvs\| + 2\|m\|_{\infty}\|\rmvs_{x}\|)
+ (\|\rho\|\|\sigma_{x}\| + 2\|\rho_{x}\|\|\sigma\|)\|\rmvs\|_{\infty} \\
&\,\,\,\,\,\,\,\, + \|\xi\|(\|\rho\|\|\rmvs_{x}\|_{\infty} + \|\rho_{x}\|\|\rmvs\|_{\infty})
+ \|\sigma\|(\|\theta\| \|\rmvs_{x}\|_{\infty} + \|\theta_{x}\|\|\rmvs\|_{\infty}) \\
&\,\,\,\,\,\,\,\, + \|\xi\|(\|\theta\|\|\rmvs_{x}\|_{\infty} + \|\theta_{x}\|\|\rmvs\|_{\infty})\ .
\end{align*} 
Therefore, taking into account \eqref{eq526}, \eqref{eq55}, \eqref{eq54}, Sobolev's inequality,
\eqref{eq511}, and the inverse inequalities we see that
\begin{equation}
\abs{\wt{w}_{1}} \leq C(h^{2r} + \|\theta\|^{2})\ .
\label{eq531}
\end{equation}
In order to estimate $\wt{w}_{2}$, letting $f = u_{x}\rmvs - u\rmvs_{x}$ and recalling 
\eqref{eq512} we have 
\[
\wt{w}_{2} =(\xi,u_{x}\rmvs-u\rmvs_{x}) = (\xi,P_{h}f)=A(\theta,P_{h}f - f) + (\theta,f - f_{xx})\ .
\]
Since $f_{xx} = (u_{xx}\rmvs - u\rmvs_{xx})_{x} = (-m\rmvs + u\theta)_{x}$, 
integration by parts gives
\[
\wt{w}_{2} = A(\theta,P_{h}f - f) + (\theta,f) + (\theta,(m\rmvs)_{x}) 
+ \tfrac{1}{2}(u_{x}\theta,\theta)\ .
\]
Hence,
\[
\abs{\wt{w}_{2}} \leq \|\theta\|_{1} \|P_{h}f - f\|_{1} + \|\theta\|\|f\|
+ \|\theta\|(\|m_{x}\|_{\infty} \|\rmvs\| + \|m\|_{\infty}\|\rmvs_{x}\|) + C\|\theta\|^{2}\ .
\]
Noting by \eqref{eq55} that $\|P_{h}f - f\|_{1} \leq Ch\|f\|_{2}$ and taking into account
\eqref{eq511} and \eqref{eq526} we see that 
\[
\abs{\wt{w}_{2}} \leq C(h^{2r-1} + \|\theta\|^{2})\ .
\]
Therefore, using this and \eqref{eq531}, \eqref{eq527}, and \eqref{eq529} give
\[
\tfrac{d}{dt}\|\theta\|^{2} \leq C(h^{2r-1} + \|\theta\|^{2})\ .
\]
Gronwall's lemma yields now that $\|\theta\|\leq Ch^{r-1/2}$, proving \eqref{eq525}.  
\hfill $\qed$                                             
\end{proof}
\textbf{Remark} It is also possible to construct a standard Galerkin method directly for the
ibvp \eqref{eqchb}. For this purpose, using the notation 
introduced in the beginning of this section we will, as an example, discretize the problem in the
space of cubic splines. i.e. in 
$S_{h}=\{\phi \in S_{h}^{4,2} : \phi(0)=\phi''(0) = \phi(1)=\phi''(1)=0\}$. We may assume now,
cf. \cite{df}, \cite{s}, \cite{adm}, that
\[
\|R_{h}\rmvs - \rmvs\|_{j} \leq Ch^{4-j}\|\rmvs\|_{4}, \quad j=0,1,2\ ,
\]
for $\rmvs \in H^{4}(0,1)$ satisfying $\rmvs(0)=\rmvs''(0)=\rmvs(1)=\rmvs''(1)=0$, and that
$\|\chi\|_{k}\leq Ch^{-1}\|\chi\|_{k-1}$, $k=1,2$, $\chi \in S_{h}$. The standard Galerkin 
semidiscrete approximation of \eqref{eqchb} in $S_{h}$ is then the map 
$u_{h} : [0,T] \to S_{h}$ defined by
\begin{equation}
A(u_{ht}, \phi) = (u_{hx}u_{hxx},\phi) - 3(u_{h}u_{hx},\phi) - (u_{h}u_{hxx},\phi'), 
\,\,\,\, t\in [0,T],\,\,\,\, \phi \in S_{h}\ ,
\label{eq532}
\end{equation}
with
\begin{equation}
u_{h}(0) = R_{h}u_{0}\ .
\label{eq533}
\end{equation}
Using standard techniques of error estimation we were able to prove only suboptimal error 
estimates for $u_{h}$. Specifically it holds that if $u$ is a classical solution of
\eqref{eqchb}, then $u_{h}$ exists for $t\in [0,T]$ and satisfies 
\begin{align*}
\max_{0\leq t\leq T}\|u -u_{h}\|_{1} & \leq Ch^{2} \\
\max_{0\leq t\leq T}\|u -u_{h}\| & \leq Ch^{2.5}\ .
\end{align*}

\section{Numerical experiments}\label{sec6}

In this section we present the results of some numerical experiments that we performed to approximate solutions of the Camassa-Holm equation, in its reduced form RCH (i.e. with $k=0$ in \eqref{eq11}). using the Galerkin finite element methods analyzed in the previous sections to discretize the equation in the spatial variable. In various places in the sequel use will be made of formulas of exact smooth travelling-wave solutions of RCH in the notation of \cite{paii} (see specifically formulas (2.29)--(2.30) in that reference): Let $\kappa>0$, $p>0$, be two parameters such that $0<\kappa p<1$, and $\tilde{c}=2\kappa^2/(1-\kappa^2p^2)$, $V=\tilde{c}+\kappa^2$. Then for $x_0\in \mathbb{R}$ the formulas
\begin{equation}\label{eq61}
u(\xi)=\kappa^2+\frac{\tilde{c}^2p^2\sech^2 \frac{1}{2}\theta}{2+\tilde{c}p^2\sech^2 \frac{1}{2}\theta}\ ,
\end{equation}
with
\begin{equation}\label{eq62}
\xi=x-Vt+x_0=\frac{\theta}{\kappa p}+\ln \left[\frac{(1+\kappa p)+(1-\kappa p)e^\theta}{(1-\kappa p)+(1+\kappa p)e^\theta} \right]\ ,
\end{equation}
define parametrically a family of single-pulse travelling-wave solutions of (RCH) that are centered at $x=x_0$, travel to the right with speed $V$ and decay exponentially as $|x|\rightarrow\infty$ to $\kappa^2$. (As a consequence of the change-of-variable formula that transforms \eqref{eq11} into the RCH \eqref{eq14}, it is seen that \eqref{eq11} with $k=\kappa^2$ possesses smooth solitary-wave solutions decaying to zero as $|x|\rightarrow\infty$, and found by replacing $u$ in \eqref{eq61} by $u+\kappa^2$ and putting $V=\tilde{c}$.) As $\kappa\rightarrow 0$ with $\kappa p\rightarrow 1$, the solution \eqref{eq61}--\eqref{eq62} of the RCH tends to the one-parameter family of peakon (soliton) solutions of the RCH given explicitly by
\begin{equation}\label{eq63}
u(x,t)=\tilde{c} \exp\left(-|x-\tilde{c}t+x_0|\right)\ .
\end{equation}
(The corresponding non-smooth travelling wave solutions of CH with $k\not=0$ are given by
$u=\tilde{u}-k$, where $\tilde{u}=\tilde{c} \exp\left(-|x-(\tilde{c}-k)t+x_0|\right)$.
%$u=\tilde{c} \exp\left(-|x-(\tilde{c}-k)t+x_0|\right)-k$). 
In the numerical experiments we compute numerically the solution \eqref{eq61}--\eqref{eq62} for each $(x,t)$, given the parameters $\kappa$, $p$, by solving \eqref{eq62} for $\theta$ by Newton's method with absolute error tolerance $10^{-10}$ and substituting in \eqref{eq61}.

\subsection{Fully discrete schemes and stability conditions}\label{sec61}

We used the classical, explicit, four-stage, fourth-order accurate Runge-Kutta scheme with a uniform time step (RK4) for the temporal discretization of the o.d.e. systems representing the various space-discrete finite element schemes of the previous sections. (Unless otherwise mentioned, we solved numerically the periodic initial-value problem for RCH.) The integrals in the finite element equations were approximated by Gauss-Legendre quadrature with five nodes in each mesh interval for cubic and quadratic splines and with three nodes for piecewise linear discretizations. The initial conditions of the fully discrete scheme were normally taken as the `elliptic' ($H^1$) projections of the initial data; taking the $L^2$ projections gave practically the same results.

Due to the presence of the $-u_{xxt}$ term in the left-hand side of RCH the o.d.e. semidiscrete systems are only mildly stiff and the fully discrete schemes were stable under a Courant number restriction. To get some indication of the maximum allowed Courant number we approximated two smooth travelling wave solutions given by \eqref{eq61}--\eqref{eq62} corresponding to $\kappa=1$ and $V=4$, and $V=6$, on $[-100,100]$ with $h=0.1$, and tested the stability of the code integrating up to $T=100$. We also studied the stability of the code for peakons given by \eqref{eq63} for $V=\tilde{c}=1$, $2$, $3$ on $[-100,100]$ with $h=0.05$. The maximum allowed values of the Courant number $V\Delta t/h$ for stability are recorded in Table \ref{T1}. They depend on $V$ in the case of smooth travelling waves but are independent of $V$ for peakons.

\begin{table}[ht!]
\begin{tabular}{c||cc|ccc}
\hline
\multicolumn{1}{c}{} & \multicolumn{2}{c}{Standard Galerkin} &  \multicolumn{3}{c}{Modified Galerkin}\\
\hline
 & cubic & quadratic & cubic & quadratic & linear  \\ \hline
Travelling wave $V=4$ &  $2.92$ & $3.91$ & $2.62$ & $2.93$ & $3.93$ \\
Travelling wave $V=6$ &  $2.18$ & $2.68$ & $1.98$ & $2.18$ & $1.79$ \\
peakons & $1.54$ & $1.83$ & $1.41$ & $1.54$ & $1.83$ \\
\hline
\end{tabular}
\caption{Maximum Courant number $V\Delta t/h$ for stability for various spatial discretizations}
\label{T1}
\end{table}

\subsection{Spatial convergence rates}\label{sec62}

We performed a series of numerical experiments with our fully discrete schemes in order to compute the numerical rates of convergence of the spatial discretizations in various norms. (In all cases we checked that the errors practically did not change when we computed with smaller values of $\Delta t$. Usually we took $\Delta t/h=1/10$.)

In the case of smooth solutions of the periodic ivp the numerical experiments confirm the theoretical $O(h^r)$ $L^2$-error convergence rates. (We experimented, on uniform meshes, with cubic and quadratic splines for the standard Galerkin method, and with cubic, quadratic, and linear splines for the modified scheme). The numerical values of the convergence rates in other norms were practically equal those expected from the approximation properties of the finite element spaces, namely $r-1$ for the $H^1$ norm in all cases, $r$ for the $L^\infty$ norm in all cases, and $r-2$ in the case of the $H^2$ norm for cubic and quadratic splines. It is noteworthy that for the same test problem the values of the two spatial discretizations (standard and modified) were very close in all norms.

In order to find the numerical convergence rates in the case of nonsmooth solutions we computed the evolution of the peakon of unit speed $u(x,t)=\exp\left(-|x-t|\right)$ on the interval $[-40,40]$ with periodic boundary conditions, up to $T=1$, using uniform meshes in space with $h=80/N$ and $\Delta t/h=1/10$, $\Delta t=1/M$. We used cubic, quadratic, and linear splines. Table \ref{T2} shows the observed errors (normalized by the corresponding norm of the exact solution at $T=1$) and convergence rates in the $L^2$ and $H^1$ norms in the case of cubic splines at $T=1$. 

\begin{table}[ht!]
\begin{tabular}{ c c || c c c c | c c c c }
 \hline
 \multicolumn{2}{c}{} &  \multicolumn{4}{c}{Standard Galerkin} &  \multicolumn{4}{c}{Modified Galerkin} \\
 \hline
$N$ & $M$ & $L^2$ error & rate & $H^1$ error & rate & $L^2$ error & rate & $H^1$ error & rate \\
 \hline
$160$ & $20$  & $1.1109\times 10^{-1}$ &  --  & $4.1633\times 10^{-1}$ &  -- &        $1.0346\times 10^{-1}$ &  --  & $4.0152\times 10^{-1}$ &  --   \\
$320$ &  $40$ & $5.1323\times 10^{-2}$ & $1.114$  & $3.1138\times 10^{-1}$ & $0.419$ & $4.6734\times 10^{-2}$ & $1.147$ & $2.9610\times 10^{-1}$ & $0.439$ \\
$640$ &  $80$ & $2.3124\times 10^{-2}$ & $1.150$  & $2.3106\times 10^{-1}$ & $0.430$ & $2.0617\times 10^{-2}$ & $1.181$ & $2.1716\times 10^{-1}$ & $0.447$ \\
$1280$ & $160$ & $1.0417\times 10^{-2}$ & $1.150$  & $1.7091\times 10^{-1}$ & $0.435$ & $9.1382\times 10^{-3}$ & $1.174$ & $1.5881\times 10^{-1}$ & $0.451$ \\
$2560$ & $320$ & $4.7544\times 10^{-3}$ & $1.132$  & $1.2626\times 10^{-1}$ & $0.437$ & $4.1283\times 10^{-3}$ & $1.146$ & $1.1600\times 10^{-1}$ & $0.453$ \\
$5120$ & $640$ & $2.2090\times 10^{-3}$ & $1.106$  & $9.3242\times 10^{-2}$ & $0.437$ & $1.9097\times 10^{-3}$ & $1.112$ & $8.4706\times 10^{-2}$ & $0.454$ \\
 \hline
\end{tabular}
\caption{$L^2$ and $H^1$ normalized errors and convergence rates at $T=1$, peakon solution, cubic splines, $h=80/N$, $\Delta t=1/M$}
\label{T2}
\end{table}

The experimental rates of the $L^2$ errors are equal to about $1.1$ for both Galerkin methods and those of the $H^1$ errors between $0.4$ and $0.5$. The errors of the modified method are slightly smaller. In Table \ref{T3} we compare the normalized $L^2$, $L^\infty$, and $H^1$ errors and convergence rates of the cubic ($r=4$), quadratics ($r=3$), and linear ($r=2$) spline discretizations for the same test problem. We show only the data for the line corresponding to $N=5120$, $M=640$). As expected from the low regularity of the solution, the rates of convergence for all spline discretizations were roughly the same for a given norm. It appears that the modified scheme with cubic splines yields the smallest errors. 

\begin{table}[ht!]
\begin{tabular}{ l || l l | l l l }
 \hline
 \multicolumn{1}{c}{} &  \multicolumn{2}{c}{Standard Galerkin} &  \multicolumn{3}{c}{Modified Galerkin} \\
 \hline
Norm & $r=4$ & $r=3$ & $r=4$ & $r=3$ & $r=2$  \\
 \hline
$L^2$ error & $2.2090\times 10^{-3}$ & $3.3557\times 10^{-3}$ & $1.9097\times 10^{-3}$ & $2.6936\times 10^{-3}$ & $3.3828\times 10^{-3}$ \\
rate & $1.106$ & $1.064$ & $1.112$ & $1.060$ & $1.125$ \\
$L^\infty$ error & $7.2834\times 10^{-3}$ & $1.1634\times 10^{-2}$ & $6.5729\times 10^{-3}$ & $7.9459\times 10^{-3}$ & $1.3519\times 10^{-2}$\\
rate & $0.902$ & $0.798$ & $0.941$ & $0.848$ & $0.814$\\
$H^1$ error & $9.3241\times 10^{-2}$ & $1.0899\times 10^{-1}$ & $8.4706\times 10^{-2}$ & $9.0104\times 10^{-2}$ & $1.1564\times 10^{-1}$ \\
rate & $0.437$ & $0.403$ & $0.454$ & $0.443$ & $0.407$\\
 \hline
\end{tabular}
\caption{$L^2$, $L^\infty$, and $H^1$ normalized errors and convergence rates at $T=1$, peakon solutions, splines of order $r$, $N=80/h=5120$, $M=1/\Delta t=640$}
\label{T3}
\end{table}

We also discretized the ibvp (C-H-s-b) for the RCH using the spatial discretization \eqref{eq58}--\eqref{eq510} coupled with RK4 in order to check the theoretically predicted convergence rates of the error bounds in \eqref{eq511} and \eqref{eq525}. The spatial discretization was effected using piecewise linear continuous functions ($r=2$) and cubic splines ($r=4$) on uniform and nonuniform meshes in several test problems. In the case of uniform mesh the $L^2$, $L^\infty$, and $H^1$ numerical convergence rates were found to be approximately equal to $r$, $r$, and $r-1$, respectively, for the approximations of both $m$ and $u$ in \eqref{eq511} and \eqref{eq525}. Thus,the predicted rates of the $L^2$ errors for m and u in \eqref{eq511} and \eqref{eq525} appear to be pessimistic for uniform meshes. In the case of a quasiuniform mesh (we experimented with the mesh $h/2,3h/2,h/2,3h/2,\cdots, h=1/N$), some results are shown in Table \ref{T4} for a test example with appropriate right-hand side and exact solution $u(x,t)=e^t[x\sin\pi x-\frac{\pi}{6}(x-1/2)+\frac{2\pi}{3}(x-1/3)^3]$, $m=u-u_{xx}$. In the case of piecewise linear functions the table shows that the apparent $L^2$-convergence rates are equal to $1$ for the approximation of $m$ and $2$ for that of $u$. (The $L^\infty$ rates, not shown in the table, were approximately equal to $1$ for $m$ and $2$ for $u$, while the $H^1$ rate for $u$ was equal to 1.) For cubic splines the observed $L^2$ rates were approximately equal to $3$ and $4$ for $m$ and $u$, respectively. (The $L^\infty$ and $H^1$ rates for $m$ were equal to approximately $2.8$ and $2$, respectively, whereas the corresponding rates for $u$ came out to be $4$ and $3$, respectively.)
\begin{table}[ht!]
\begin{tabular}{cc}
(a) $r=2$ & (b) $r=4$ \\
\begin{tabular}{c||cc|cc}
\hline
 \multicolumn{1}{c}{} &  \multicolumn{2}{c}{$m$} &  \multicolumn{2}{c}{$u$} \\
\hline
$N$ & $L^2$ error & rate & $L^2$ error & rate \\ \hline
$32$ & $7.2147\times 10^{-1}$ & -- & $1.2979\times 10^{-3}$ & --\\
$64$ & $3.9516\times 10^{-1}$ & $0.866$ & $3.2705\times 10^{-4}$ & $1.989$ \\
$128$ & $2.1177\times 10^{-1}$ & $0.900$ & $8.1585\times 10^{-5}$ & $2.003$ \\
$256$ & $1.0508\times 10^{-1}$ & $1.011$ & $2.0389\times 10^{-5}$ & $2.000$ \\
$512$ & $5.2300\times 10^{-2}$ & $1.007$ & $5.0977\times 10^{-6}$ & $2.000$ \\
$1024$ & $2.6110\times 10^{-2}$ & $1.002$ & $1.2745\times 10^{-6}$ & $2.000$ \\
$2048$ & $1.3048\times 10^{-2}$ & $1.001$ & $3.1862\times 10^{-7}$ & $2.000$ \\
\hline
\end{tabular}
&
\begin{tabular}{c||cc|cc}
\hline
 \multicolumn{1}{c}{} &  \multicolumn{2}{c}{$m$} &  \multicolumn{2}{c}{$u$} \\
\hline
$N$ & $L^2$ error & rate & $L^2$ error & rate \\ \hline
$8$ & $3.9055\times 10^{-2}$ & -- & $3.8715\times 10^{-4}$ & --\\
$16$ & $6.0719\times 10^{-3}$ & $2.685$ & $3.4546\times 10^{-5}$ & $3,486$ \\
$32$ & $8.2945\times 10^{-4}$ & $2.872$ & $2.5099\times 10^{-6}$ & $3.783$ \\
$64$ & $1.0748\times 10^{-4}$ & $2.948$ & $1.6634\times 10^{-7}$ & $3.915$ \\
$128$ & $1.3568\times 10^{-5}$ & $2.986$ & $1.0654\times 10^{-8}$ & $3.965$ \\
$256$ & $1.7037\times 10^{-6}$ & $2.993$ & $6.7326\times 10^{-10}$ & $3.984$ \\
$512$ & $2.1325\times 10^{-7}$ & $2.998$ & $4.2300\times 10^{-11}$ & $3.992$ \\
$1024$ & $2.6665\times 10^{-8}$ & $3.000$ & $2.7035\times 10^{-12}$ & $3.968$ \\
\hline
\end{tabular}
\end{tabular}
\caption{$L^2$ errors and rates of convergence for $m$ and $u$ for the ibvp (CH-s-b), quasiuniform mesh, (a) linear splines, (b) cubic splines}
\label{T4}
\end{table}
Thus, it appears that for our quasiuniform mesh examples the rates in \eqref{eq511} are sharp, while the rate $r-1/2$ in \eqref{eq525} is pessimistic by $1/2$. (The reduction of order for the Galerkin approximation of $m$ is not surprising since $m$ is the solution of a first-order hyperbolic p.d.e.; the $L^2$ rate of the approximation of $u$ is expected to be optimal as $u_h$ is the Galerkin approximation of a two-point b.v.p. whose right-hand side $m_h$ is $C^0$ if $r=2$ and $C^2$ if $r=4$.)

\subsection{Conservation of invariants}\label{sec63}

It is well known, cf. \cite{ch}, and easy to check, that solutions of the Cauchy problem of the CH equation \eqref{eq11} and also of the periodic ivp for \eqref{eq11} on an interval $(a,b)$ satisfy the conservation laws $\frac{d}{dt}H_i[u]=0$, $i=0,1,2$, where
\begin{equation}\label{eq64}
H_0[u]=\int_a^bu~dx,\quad H_1[u]=\int_a^b(u^2+u_x^2)~dx,\quad H_2[u]=\int_a^bu(u^2+u_x^2+2ku)~dx\ .
\end{equation}
In the case of the periodic ivp for the CH in the system form \eqref{eq13} it holds that $\frac{d}{dt}\tilde{H}_i[m,u]=0$, $i=0,1,2$, where
\begin{equation}\label{eq65}
\tilde{H}_0[m,u]=\int_a^b m~dx,\quad \tilde{H}_1[m,u]=\int_a^b mu~dx,\quad \tilde{H}_2[m,u]=\int_a^b(u^2m-uu_x^2+2ku^2)~dx\ .
\end{equation}
In order to assess the quality of the numerical approximations, it is important to check the extent to which the fully discrete schemes under study preserve these invariants. (As usual, we consider the RCH equation). The standard Galerkin semidiscretization \eqref{eq31} clearly preserves $H_0$ and $H_1$ in the sense that $H_i[u_h(t)]=H_i[u_h(0)]$, $i=0,1$. The invariant $H_0$ is also trivially preserved by the full discretization of \eqref{eq31} by the RK4 scheme, in contrast to $H_1$ which is not preserved. The modified Galerkin method \eqref{eq41}--\eqref{eq42} for the periodic ivp preserves $\tilde{H}_1$: Let $m_h$, $u_h$ be the solution of \eqref{eq41}--\eqref{eq44}. Then, using \eqref{eq42} with $\phi=u_h$, and \eqref{eq41} with $\phi={u_h}_t$ we have
\begin{align*}
\frac{d}{dt}\tilde{H}_1[m_h,u_h] &= ({m_h}_t,u_h)+(m_h,{u_h}_t)\\
&= -\left((m_hu_h)_x,u_h \right)-(m_h{u_h}_x,u_h)+(u_h,{u_h}_t)+({u_h}_x,{u_h}_{xt})\\
&= \frac{1}{2}\frac{d}{dt}\|u_h\|_1^2\ .
\end{align*}
On the other hand \eqref{eq41} with $\phi=u_h$ gives $(m_h,u_h)=\|u_h\|_1^2$, i.e. $\frac{d}{dt}\tilde{H}_1[m_h,u_h]=\frac{d}{dt}\|u_h\|_1^2$. From these identities we conclude that $\frac{d}{dt}\tilde{H}_1[m_h,u_h]=\frac{d}{dt}\|u_h\|_1^2=0$, as claimed. Note that the temporal discretization of \eqref{eq41}--\eqref{eq44} with RK4 does not conserve $\tilde{H}_1$.

In what follows we present some graphs of the temporal evolution of the invariants using the fully discrete schemes. We computed the quantities 
\begin{equation*}
E_i(t)=\log_{10}\left|\frac{H_i(t)-H_i(0)}{H_i(t)}\right|\ ,
\end{equation*}
(and the analogous $\tilde{H}_i$ in the case of the modified method) for $t\geq \Delta t$. Here, $H_i(t)=H_i\left(u_h(t)\right)$ and $\tilde{H}_i(t)=\tilde{H}_i\left[m_h(t),u_h(t)\right]$.
In the case of the {\em smooth solutions} we solved numerically the periodic ivp for the RCH on the spatial interval $[-50,50]$ up to $t=100$ with the standard and modified method using as initial condition $u_0(x)=1+e^{-x^2}$.
\begin{figure}[ht!]
  \centering
  \includegraphics[width=0.7\columnwidth]{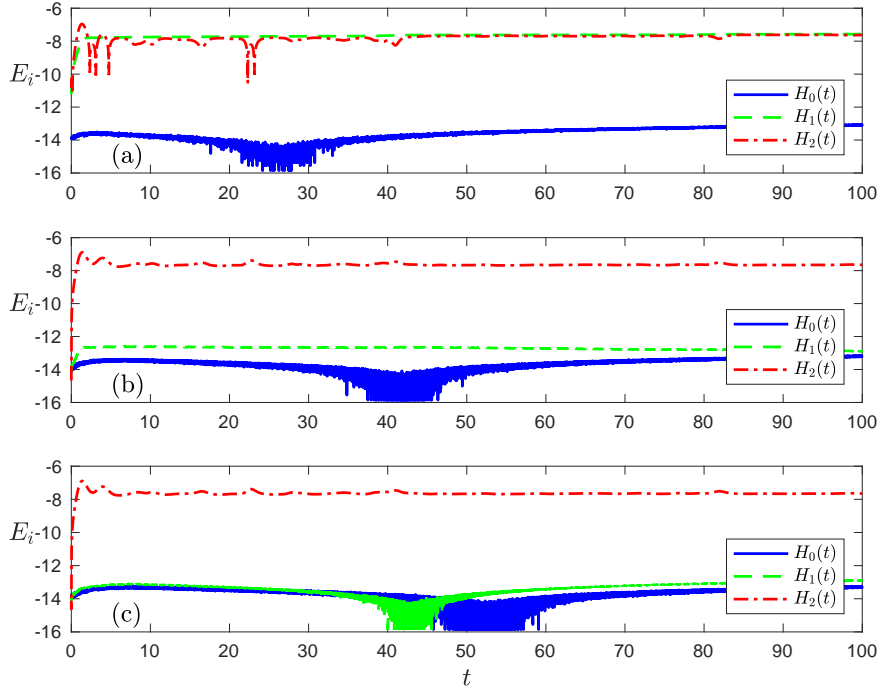}
  \caption{Preservation of invariants $H_0$, $H_1$, $H_2$, smooth solutions. Standard Galerkin method with cubic splines, $h=0.1$, (a) $\frac{\Delta t}{h}=\frac{1}{10}$, (b) $\frac{\Delta t}{h}=\frac{1}{100}$, (c) $\frac{\Delta t}{h}=\frac{1}{200}$}
  \label{fig1}
\end{figure}

\begin{figure}[ht!]
  \centering
  \includegraphics[width=0.7\columnwidth]{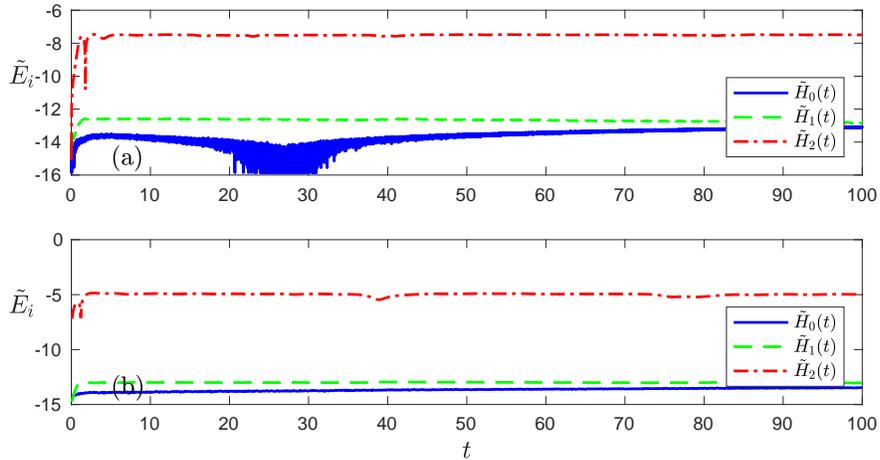}
  \caption{Preservation of invariants $\tilde{H}_0$, $\tilde{H}_1$, $\tilde{H}_2$, smooth solutions. modified Galerkin method, $h=0.1$, $\Delta t=10^{-3}$ (a) cubic splines, (b) p.w. linear continuous functions}
  \label{fig2}
\end{figure}

\begin{figure}[ht!]
  \centering
  \includegraphics[width=0.7\columnwidth]{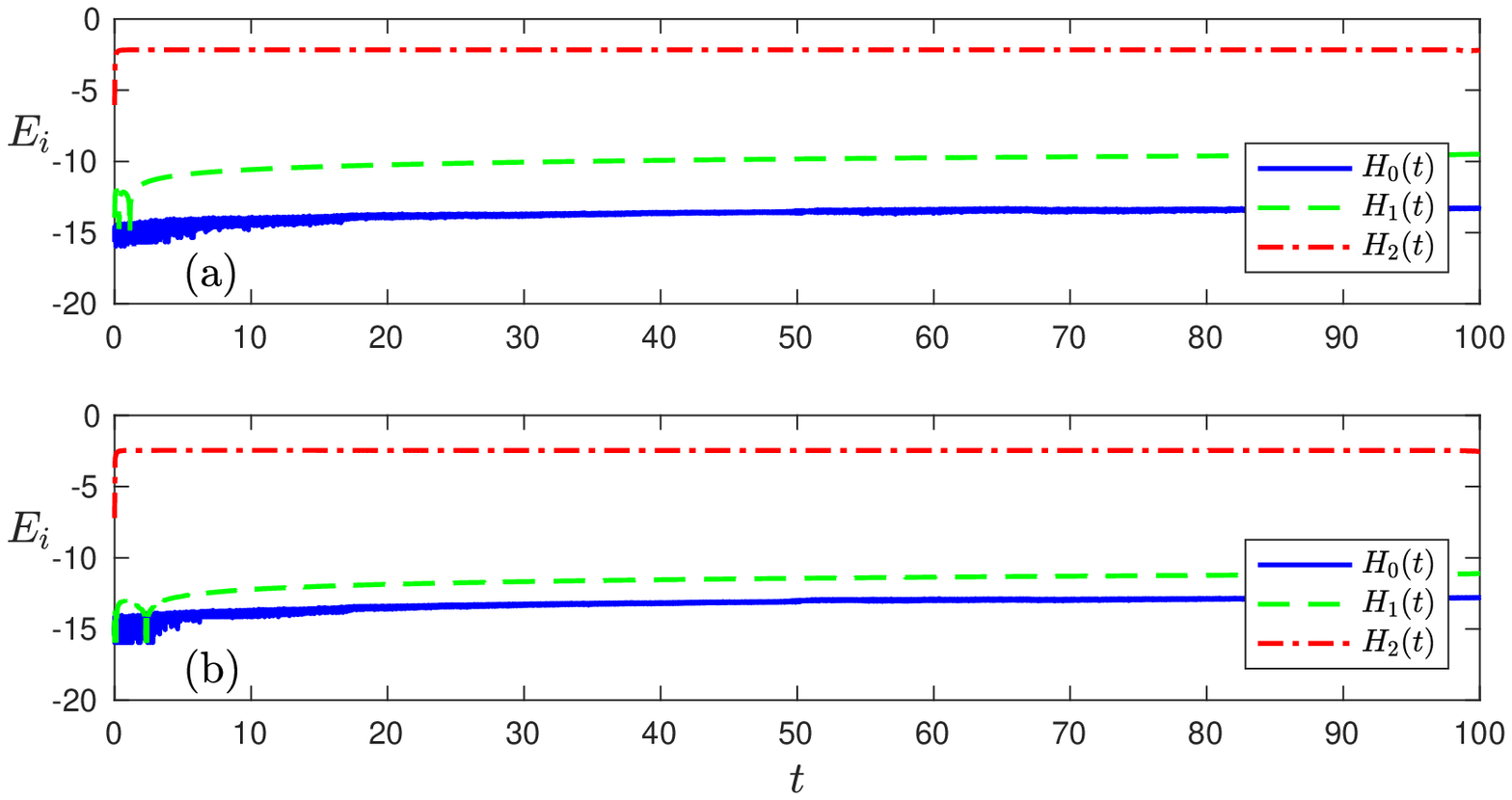}
  \caption{Preservation of invariants $H_0$, $H_1$, $H_2$, peakon, $A=1$. Standard Galerkin method with cubic splines, (a) $h=0.1$, $\Delta t=10^{-3}$, (b) $h=0.05$, $\Delta t=2.5\times 10^{-4}$}
  \label{fig3}
\end{figure}

Figure \ref{fig1} shows the graphs of the logarithmic relative errors $E_i$ of $H_i$ as functions of $t$ for this example, produced for $h=0.1$ and diminishing $\Delta t/h$ by the standard Galerkin method with cubic splines coupled with RK4 time stepping. $H_0$ is preserved to roundoff as expected, while $H_2$ is preserved to about $8$ decimal digits for this value of $h$. (The error in $H_2$ is practically due to the spatial approximation since it remains almost constant as $\Delta t$ is decreased.) The error in $H_1$ also diminishes with $\Delta t$; for $\frac{\Delta t}{h}=\frac{1}{200}$ its temporal discretization part is practically negligible and $H_1$ is preserved to roundoff as expected.

In the case of the modified Galerkin fully discrete scheme the cubic spline results for the preservation of the $H_i=H_i[u_h]$ are practically the same with those shown in Figure 1. In Figure \ref{fig2} we show the evolution of the logarithmic errors $\tilde{E}_i$ of the invariants $\tilde{H}_i[m_h,u_h]$ pertinent to the system form of RCH for the fully discrete schemes with cubic splines and piecewise linear functions (the latter are admitted by the modified scheme), for the same example, with $h=0.1$ and $\Delta t=10^{-3}$. For cubic splines $\tilde{H}_2$ is preserved to at least $7$ decimal digits and to about $5$ digits for the less accurate discretization with piecewise linear functions. For both discretizations the relative error of $\tilde{H}_1$ is of $O(10^{-13})$ and is due to the temporal discretization as $\tilde{H}_1$ is preserved for both semidiscrete schemes. Apparently $\tilde{H}_0$ seems to be preserved to almost roundoff in both cases.
In the case of {\em peakons} we computed the logarithmic errors of the invariants as functions of $t$ in several test runs. In Figure \ref{fig3}, as an example, we show the results in the case of a peakon of unit amplitude and speed approximated by the standard Galerkin method with cubic splines for our usual (a), and a finer mesh (b). We observe that $H_0$ is preserved to roundoff, that the error of $H_1$ is small and is reduced by at most two digits in (b), while that of $H_2$ practically remains the same. This is due of course to the lack of smoothness of the solution that reduces the rate of convergence of the schemes, cf. Section \ref{sec62}.

\subsection{Error indicators for travelling waves}\label{sec64}

In order to further assess the accuracy of our numerical methods we studied several error indicators pertinent to travelling-wave solutions, namely the {\em amplitude}, {\em phase}, {\em speed}, and {\em shape} errors, \cite{bps}, \cite{bdkmc}, for smooth travelling waves and peakons. Since the amplitude of the travelling wave remains constant, we define the (normalized) {\em amplitude error} at $t=t^n=n\ \Delta t$ as 
\begin{equation}\label{eq66}
E_{amp}(t^n)=\frac{|U(x^\ast(t^n),t^n)-U_0|}{|U_0|}\ ,
\end{equation}
where $U(x,t^n)$ is the fully discrete numerical approximation at $t^n$, $U_0$ is the initial amplitude (peak value) of the exact travelling wave and $x^\ast(t^n)$ is the point where $U(\cdot,t^n)$ achieves its maximum. The value of $x^\ast(t^n)$ is found by solving the equation $\frac{d}{dx} U(x,t^n)=0$ by Newton's method (with tolerance $10^{-10}$) for smooth solutions and by the bisection method for peakons. (For piecewise linear functions we took as $x^\ast$ the mesh point where $max_i U(x_i,t^n)$ occurs.) As initial value for the Newton iteration we took the quadrature node at which $U$ has a discrete maximum. We started the bisection method by taking an interval of length $4h$ around that quadrature node.

Finding $x^\ast(t^n)$ enables us to compute a {\em phase error} of the numerical solution $U$ at $t^n$ as
\begin{equation}\label{eq67}
E_{phase}(t^n)=|x^\ast(t^n)-Vt^n|\ ,
\end{equation}
for a travelling wave of speed $V$, initially centered at $x=0$, and a {\em speed error} defined as
\begin{equation}\label{eq68}
E_{speed}(t^n)=|V-V_h(t^n,\tau)|\ ,
\end{equation}
where $V_h(t^n,\tau)=\left(x^\ast(t^n)-x^\ast(t^n-\tau)\right)/\tau$. We usually take $\tau$ much larger than $\Delta t$ to smooth out oscillations in the discrete speed $V_h(t^n,\tau)$. (In the computations to be reported in the sequel on the temporal interval $[0,100]$, we normally take $\tau=1$.)

Finally, the (normalized) {\em shape error} measures the amount by which the numerical solution at $t^n$ differs from the exact travelling wave profile $u(\cdot,t^n)$ traslated so as to give the best fit. It is defined with respect to the $L^2$ norm as
\begin{equation}\label{eq69}
E_{shape}(t^n)=\min_s\zeta(s), \quad \zeta(s):=\frac{\|U(\cdot,t^n)-u(\cdot,s)\|}{\|u(\cdot,0)\|}\ ,
\end{equation}
where the minimum of $\zeta(s)$ is found again using the Newton or the bisection method to solve $\frac{d}{ds}\zeta^2(s)=0$ in the vicinity of $t^n$.

In Table \ref{T5} we show the values of the amplitude, phase and shape error at $t^n=100$ of the fully discrete approximations produced by the standard and the modified Galerkin method in the case of a smooth travelling wave of the form \eqref{eq61}--\eqref{eq62} with $\kappa=1$ and $V=4.333$. We took the spatial interval $[-100,100]$ and computed with $h=0.1$, $0.05$, $0.025$ and $\Delta t=h/10$.

\begin{table}[ht!]
\begin{tabular}{ l || l l l  | l l l  }
 \multicolumn{7}{l}{(a) Standard Galerkin} \\
 \hline
 \multicolumn{1}{c}{} &  \multicolumn{3}{c}{$r=4$, cubic splines} &  \multicolumn{3}{c}{$r=3$, quadratic splines} \\
 \hline
$h$ & $E_{amp}$ & $E_{phase}$ & $E_{shape}$ & $E_{amp}$ & $E_{phase}$ & $E_{shape}$  \\
 \hline
$0.1$ & $9.1617\times 10^{-9}$ & $7.0771\times 10^{-6}$ & $1.2058\times 10^{-8}$ & $5.4368\times 10^{-7}$ & $2.6859\times 10^{-5}$ & $1.0699\times 10^{-7}$ \\
$0.05$ & $5.5416\times 10^{-10}$ & $3.1641\times 10^{-7}$ & $5.6834\times 10^{-10}$ & $3.2712\times 10^{-8}$ & $1.5455\times 10^{-6}$ & $6.6264\times 10^{-9}$\\
$0.025$ & $6.7388\times 10^{-11}$ & $1.6421\times 10^{-8}$ & $3.5359\times 10^{-11}$ & $2.0090\times 10^{-9}$ & $9.0883\times 10^{-8}$ & $4.1343\times 10^{-10}$ \\
 \hline
\end{tabular}
\begin{tabular}{ l || l l l  | l l l  }
 \multicolumn{7}{l}{(b) Modified Galerkin} \\
\hline
 \multicolumn{1}{c}{} &  \multicolumn{3}{c}{$r=4$, cubic splines} &  \multicolumn{3}{c}{$r=3$, quadratic splines} \\
 \hline
$h$ & $E_{amp}$ & $E_{phase}$ & $E_{shape}$ & $E_{amp}$ & $E_{phase}$ & $E_{shape}$ \\
 \hline
$0.1$ & $8.6377\times 10^{-9}$ & $7.0627\times 10^{-6}$ & $1.2004\times 10^{-8}$ & $2.6626\times 10^{-7}$ & $1.2173\times 10^{-5}$ & $2.9430\times 10^{-7}$ \\
$0.05$ & $5.5280\times 10^{-10}$ & $3.1597\times 10^{-7}$ & $5.6723\times 10^{-10}$ & $1.5428\times 10^{-8}$ & $6.3028\times 10^{-7}$ & $1.7453\times 10^{-9}$\\
$0.025$ & $7.1619\times 10^{-11}$ & $1.6134\times 10^{-8}$ & $3.5361\times 10^{-11}$ & $9.2885\times 10^{-10}$ & $3.5562\times 10^{-8}$ & $1.0862\times 10^{-10}$ \\
 \hline
\end{tabular}
\begin{tabular}{ l || l l l}
 \multicolumn{1}{c}{} &  \multicolumn{3}{c}{$r=2$, linear splines} \\
 \hline
$h$ & $E_{amp}$ & $E_{phase}$ & $E_{shape}$ \\
 \hline
$0.1$ & $4.0487\times 10^{-5}$ & $1.7600\times 10^{-3}$ & $6.7965\times 10^{-5}$ \\
$0.05$ & $2.7453\times 10^{-6}$ & $3.3000\times 10^{-3}$ & $1.6249\times 10^{-5}$ \\
$0.025$ & $2.2631\times 10^{-7}$ & $9.5238\times 10^{-4}$ & $4.0635\times 10^{-6}$ \\
 \hline
\end{tabular}
\caption{$E_{amp}$, $E_{phase}$, $E_{shape}$ at $T=100$, mean value over the temporal interval $[80,100]$ for linear splines $E_{phase}$ (a) standard Galerkin, (b) modified Galerkin, smooth travelling wave, $\kappa=1$, $V=4.333$ }
\label{T5}
\end{table}

As expected, cubic splines ($r=4$) give the best results, while there is not much difference between the standard and the modified methods for $r=4$ and $r=3$. Over time, all these errors oscillate somewhat about mean values that remain practically constant in $t$ for $h=0.05$ and $h=0.025$.

The speed errors were practically the same for both Galerkin methods with cubic and quadratic splines for the same value of $h$ and diminished of course with $h$. The number of conserved digits of the speed was equal to five for $h=0.1$, six for $h=0.05$ and seven for $h=0.025$. In the case of piecewise linear functions with the modified method the numerical speed preserved one correct digit for $h=0.1$ and two for $h=0.05$ and $h=0.025$.

In Table \ref{T6} we show the analogous errors produced by the numerical approximations with the modified method of a peakon of speed $V=1.333$. We took $h=0.05$, $0.025$ and $0.01$ on the interval $[-100,100]$ with $\Delta t=h/10$ and show the errors at $t^n=100$.

\begin{table}[ht!]
\begin{tabular}{ l || l l l  | l l l  }
\hline
 \multicolumn{1}{c}{} &  \multicolumn{3}{c}{$r=4$, cubic splines} &  \multicolumn{3}{c}{$r=3$, quadratic splines} \\
 \hline
$h$ & $E_{amp}$ & $E_{phase}$ & $E_{shape}$ & $E_{amp}$ & $E_{phase}$ & $E_{shape}$  \\
 \hline
$0.05$ & $1.1717\times 10^{-2}$ & $6.4696\times 10^{-1}$ & $2.5744\times 10^{-2}$ & $1.6177\times 10^{-2}$ & $1.0482\times 10^{0}$ & $1.1215\times 10^{-2}$ \\
$0.025$ & $6.1999\times 10^{-3}$ & $3.2055\times 10^{-1}$ & $1.3246\times 10^{-2}$ & $8.1867\times 10^{-3}$ & $5.1723\times 10^{-1}$ & $6.6000\times 10^{-3}$\\
$0.01$ & $2.6237\times 10^{-3}$ & $1.2855\times 10^{-1}$ & $5.9284\times 10^{-3}$ & $3.4618\times 10^{-3}$ & $2.0548\times 10^{-1}$ & $3.2406\times 10^{-3}$ \\
 \hline
\end{tabular}
\begin{tabular}{ l || l l l}
 \multicolumn{1}{c}{} &  \multicolumn{3}{c}{$r=2$, linear splines} \\
 \hline
$h$ & $E_{amp}$ & $E_{phase}$ & $E_{shape}$ \\
 \hline
$0.05$ & $1.1487\times 10^{-2}$ & $6.5000\times 10^{-1}$ & $5.8839\times 10^{-2}$ \\
$0.025$ & $7.1264\times 10^{-3}$ & $3.0000\times 10^{-1}$ & $3.2941\times 10^{-2}$ \\
$0.01$ & $3.5565\times 10^{-3}$ & $1.1000\times 10^{-1}$ & $1.3585\times 10^{-2}$ \\
 \hline
\end{tabular}
\caption{$E_{amp}$, $E_{phase}$, $E_{shape}$ at $T=100$, modified Galerkin method, peakon, $V=1.333$ }
\label{T6}
\end{table}

The errors of the standard Galerkin method (for $r=4$ and $r=3$) were roughly the same. The errors did not oscillate in time and practically remained constant for the smaller values of $h$. Because the maximum of the the maximum of the peakon for linear splines coincides with the mesh node $x=33.3$ at $T=100$ the actual $E_{phase}$ happened to be 0. For this reason we present the mean value of the $E_{phase}$ for all times $t^n\in [80,100]$. In all cases the speed was conserved to two digits when $h=0.05$ and $h=0.025$, and to three digits when $h=0.01$.

\subsection{Approximation of the generation and interactions of peakons}\label{sec65}

In closing we record some observations concerning the approximation of peakons by the fully discrete Galerkin methods under consideration:

(i) When a nonsmooth function such as a peakon is used as initial value in an evolution numerical experiment one expects a practically localized oscillatory error to appear at $t=0$ as a result of projecting the peakon onto the finite element space. As previously mentioned we normally use as initial value of the discrete schemes the $H^1$ projection of the initial data; using the $L^2$ projection or the interpolant leads to similar errors. 

As an example, consider the evolution shown in Figure \ref{fig4}, in which a peakon of unit speed, initially centered at $x=0$, is approximated by the $H^1$ projection in the space, of quadratic splines and its evolution is effected by the fully discrete modified Galerkin method up to $T=50$. An oscillatory error, decaying fast in space, of amplitude $7.5\times 10^{-3}$ for $h=0.05$, is generated near $x=0$ where the initial peakon was located. It diminishes slowly as $h$ is decreased; for example its amplitude is about $1.6\times 10^{-3}$ when $h=6.25\times 10^{-3}$. A similar error is observed in the case of the standard Galerkin method for the other spline spaces; the modified method gives somewhat better results. We observed that the same type of error appears in spectral discretizations of the CH equation with the interpolant e.g. taken as initial condition.
\begin{figure}[ht!]
  \centering
  \includegraphics[width=0.7\columnwidth]{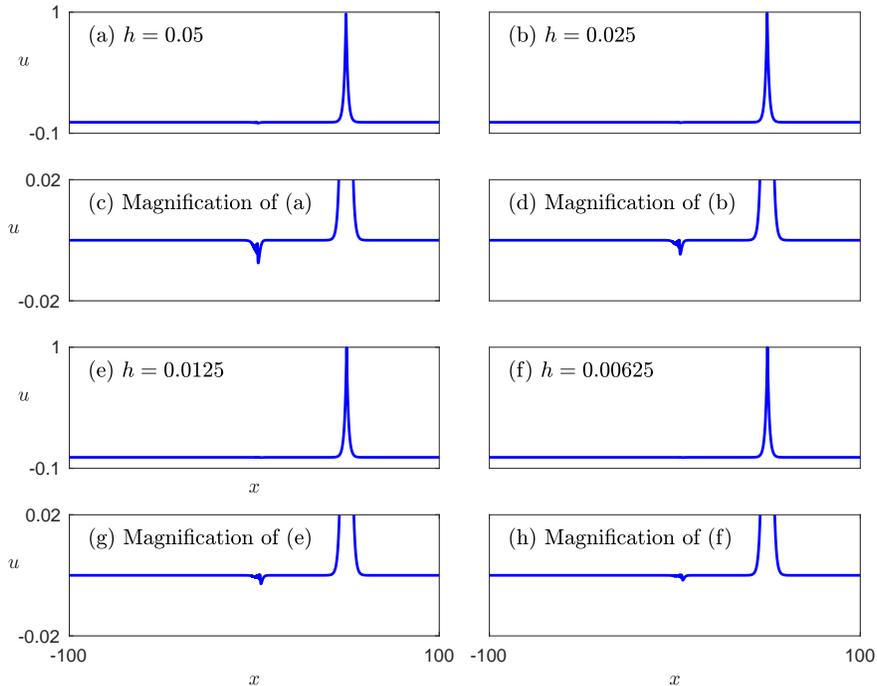}
  \caption{Peakon propagation with the modified Galerkin method and quadratic splines for various values of $h$, $\Delta t=h/10$, $T=50$}
  \label{fig4}
\end{figure}

\begin{figure}[ht!]
  \centering
  \includegraphics[width=0.7\columnwidth]{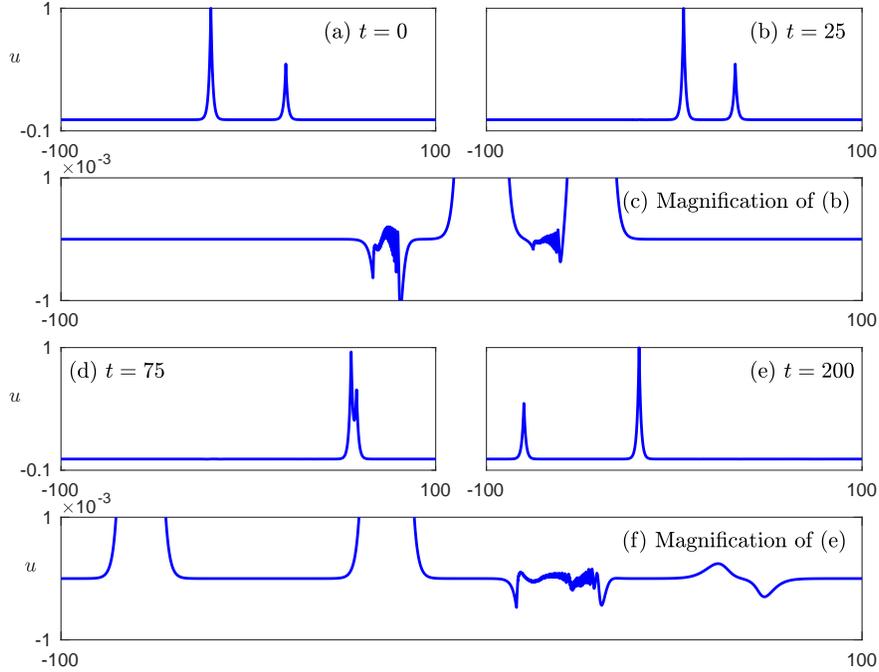}
  \caption{Interaction of two peakons, modified Galerkin, cubic splines, $h=0.005$, $\Delta t=h/10$}
  \label{fig5}
\end{figure}

\begin{figure}[ht!]
  \centering
  \includegraphics[width=0.7\columnwidth]{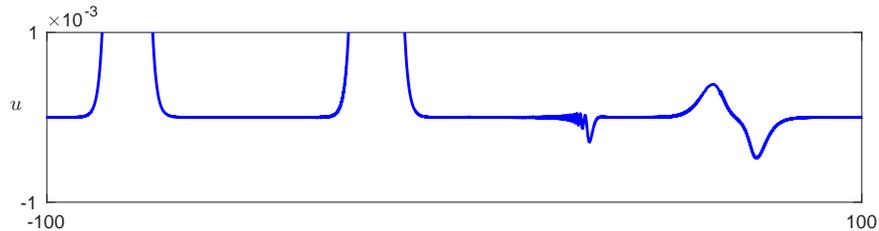}
  \caption{Interaction of two peakons (evolution of Figure \ref{fig5}) using the standard pseudospectral-RK4 method, $N=16384$, $\Delta=5\times10^{-3}$}
  \label{fig6}
\end{figure}

This small noise remains stationary in time in the vicinity of the spatial point of its generation. Apparently this property is shared by the CH equation (due to the presence of the $-u_{xxt}$ term) with the BBM equation, for which, it was shown in \cite{BBM} that initial discontinuities in the data do not propagate. (Similar errors, noticed in numerical solutions of the BBM equation in \cite{bps}, were caused by the truncation of the decaying ends of an initial solitary-wave profile in that reference.)

(ii) Peakons, being solitons, interact elastically, and their overtaking collisions have been described in detail analytically cf. e.g. \cite{ce1,p}. However, when their interactions are simulated by a numerical evolution method one observes that numerical artifacts, in the form of small residual waves, are produced.

As an example, we consider two peakons of amplitudes $1.0$ and $0.5$ centered initially at $x=-20$ and $20$, respectively. With those initial values we integrate the RCH with periodic conditions on the interval $[-100,100]$ using the fully discrete modified Galerkin method with cubic splines with $h=5\times 10^{-3}$ and $\Delta t=h/10$. The results of the numerical evolution up to $t=200$ are shown in Figure \ref{fig5}.

In Figure \ref{fig5}(c), at $t=25$, we observe again the small oscillations (stationary in the vicinity of $x=-20$ and $x=20$) produced by the $H^1$ projection of the initial peakons in the finite element space and previously commented on. The peakons first interact in a spatial and temporal window centered at $(x,t)=(60,80)$ and wrap around the boundary due to periodicity. At $t=200$ (Figure \ref{fig5}(e), (f)) we observe the high frequency stationary oscillations produced near $x=20$ due to the initial approximation of the smaller peakon. (The analogous oscillations produced near $x=-20$ by the projection of the larger peakon have `climbed' on the larger peakon which is centered at $x=-20$ at this time instant.) We also observe near $x=60$ a large-wavelength {\em wavelet} of amplitude about $2.8\times 10^{-4}$. This wavelet is stationary; note that the dispersion relation of the linearized RCH is $\omega=0$, and therefore small-amplitude disturbances are not expected to propagate. The wavelet is a numerical artifact produced by the numerical approximation of the peakon collision. Its amplitude diminishes slowly as $h$ is decreased; e.g. it was equal to $5\times 10^{-4}$ when $h=0.01$. The wavelet reappeared near $x=60$ after the next interaction (due to periodicity), of the peakons, that occurred around $t=480$. (In order to observe the wavelet after the subsequent interactions, which all occur near $x=60$, one has to `clean' smoothly the solution profile well after the first interaction from small-amplitude artifacts and let the main pulses interact again). That the wavelet is specific to the numerical interaction of peakons may also be seen from the fact that it is absent when overtaking collisions of {\em smooth} travelling waves of the RCH are simulated numerically. When we did so (with the same numerical method taking $h=0.1$ and $\Delta t=h/10$), we observed that only the expected numerical dispersive tail was observed; the latter had an amplitude of $O(10^{-6})$ and diminished fast as $h$ was decreased. 

The wavelet also appears after numerical peakon interactions effected by other type of discretizations. For example, the pseudospectral-RK4 method with $N=16,384$, $\Delta t=5\times 10^{-3}$ gives for the analogous experiment shown in Figure \ref{fig5}, at $t=200$ the profile given in Figure \ref{fig6}, very similar to the one of Figure \ref{fig5}(f). The wavelet and the initial oscillations diminish slowly as $N$ is increased and $\Delta t$ is decreased. The fact that the numerical interaction of peakons with spectral methods give small errors that diminish slowly as the mesh is refined was already observed in \cite{kl}. Here we have pointed out that the small errors are of the two kinds depicted in Figures \ref{fig5} and \ref{fig6}.

\begin{figure}[ht!]
  \centering
  \includegraphics[width=0.7\columnwidth]{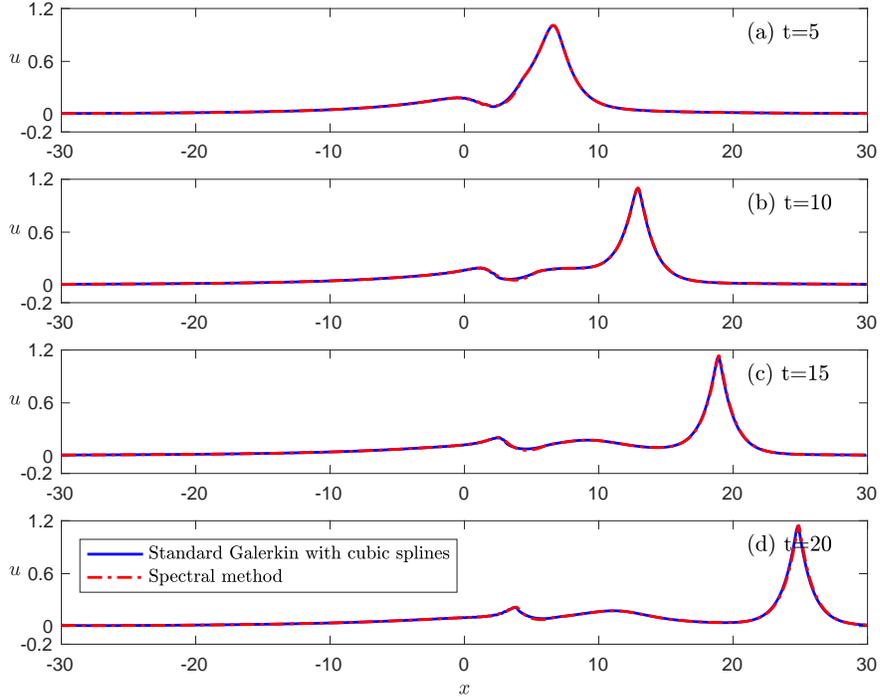}
  \caption{Evolution of RCH with initial profile \eqref{eq610}, standard Galerkin method with cubic splines and pseudospectral-RK4 scheme}
  \label{fig7}
\end{figure}

\begin{figure}[ht!]
  \centering
  \includegraphics[width=0.7\columnwidth]{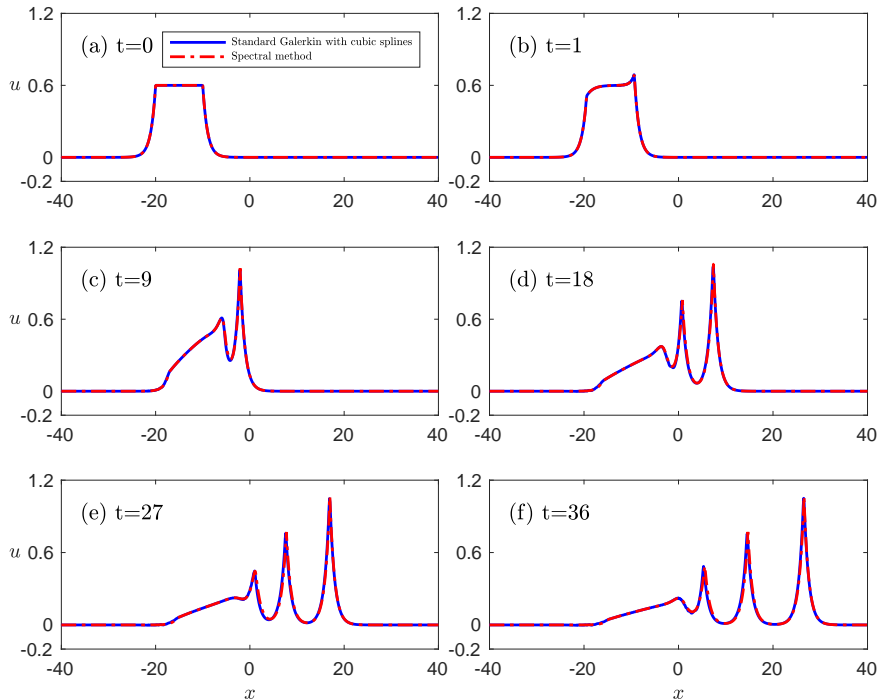}
  \caption{Evolution of RCH with initial profile \eqref{eq611}, Standard Galerkin method with cubic splines and pseudospectral-RK4}
  \label{fig8}
\end{figure}

(iii) We finally present results obtained by applying our numerical methods to two evolution problems for RCH that have been proposed in the literature as examples of peakon generation from continuous initial profiles that have discontinuous derivatives. The first, \cite{hri,xs,clp2}, corresponds to the initial value
\begin{equation}\label{eq610}
u_0(x)=\frac{10}{(3+|x|)^2}, \quad x\in [-30,30]\ ,
\end{equation}
which we integrated up to $t=20$ by our fully discrete standard Galerkin method with cubic splines using $h=0.1$, $\Delta t=0.001$, cf. Figure \ref{fig7}.
Superimposed on the spline graphs of Figure \ref{fig7} are the corresponding obtained by the usual pseudospectral-RK4 discretization of RCH with $N=4096$, $\Delta t=0.01$. (For smaller values of $N$ the pseudospectral method had a noticeable phase error by $t=20$.)

The second initial profile, \cite{as,xs}, corresponds to the plateau function given by
\begin{equation}\label{eq611}
u_0(x)=\left\{\begin{array}{lcl}
c\ e^{x+5} &, & x\leq -5\\
c &, & |x|\leq 5\\
c\ e^{-x+5} &, & x\geq 5
\end{array}\right.\quad x\in[-40,40],\quad c=0.6\ ,
\end{equation}

and produces the evolution shown up to $t=36$ in Figure \ref{fig8}. For the fully discrete stasdard Galerkin scheme with cubic splines  we took $h=0.1$ and $\Delta t=0.01$. Superimposed are shown the corresponding profiles generated by the usual pseudospectral-RK4 scheme for $N=8192$, $\Delta t=0.001$. For smaller values of $N$ the scheme exhibited phase and amplitude errors in approximating the larger emerging peakon. In both test problems the Galerkin method gave very accurate results.

\bibliographystyle{plain}
\bibliography{biblio}
\bigskip

\end{document}